\documentclass[oneside,a4paper]{article} %
      \usepackage{color}

\usepackage{amsmath,amsthm,amssymb}
\usepackage[T1]{fontenc}
\usepackage[utf8]{inputenc}
\DeclareUnicodeCharacter{00A0}{ }
\usepackage[a4paper]{geometry}

   \usepackage{hyperref} %
  \numberwithin{equation}{section}

 \title{Three representations of the fractional $p$-Laplacian: \\[3pt] semigroup, extension and Balakrishnan formulas}
 \author{Félix del Teso$^1$, David Gómez-Castro$^2$, Juan Luis Vázquez$^3$}
\date{}

 \usepackage[nameinlink]{cleveref}
 \usepackage{enumerate}
 \usepackage{amsmath}
 \usepackage{mathtools}
 \usepackage[utf8]{inputenc}

 \newcommand{\defeq}{\vcentcolon=}

 \DeclareMathOperator{\N}{{\mathbb N}}

 \newcommand{\Rd}{{\mathbb R^n}}

 \newcommand{\fpl}{{(-\Delta)_p^s}}

 \DeclareMathOperator{\dist}{dist}
 \DeclareMathOperator{\PV}{PV}

 \newcommand*\diff{\mathop{}\!\mathrm{d}}
 \newcommand{\vertiii}[1]{{\left\vert\kern-0.25ex\left\vert\kern-0.25ex\left\vert #1
 		\right\vert\kern-0.25ex\right\vert\kern-0.25ex\right\vert}}

 	\newcommand{\R}{\mathbb{R}}
 	\newcommand{\Z}{\mathbb{Z}}
 	
 	\newcommand{\veps}{\varepsilon}
 	\newcommand{\dd}[1]{\diff{#1}}
 	
 	\DeclareUnicodeCharacter{00A0}{ }

 \usepackage{amsthm}

 \newtheorem{theorem}{Theorem}[section]
 \newtheorem{lemma}[theorem]{Lemma}%

 \theoremstyle{definition}
 \newtheorem{remark}[theorem]{Remark}%

\begin{document}

\maketitle

 \begin{abstract}
We introduce three representation formulas for the fractional $p$-Laplace operator in the whole range of parameters $0<s<1$ and $1<p<\infty$.  Note that for $p\ne 2$ this a nonlinear operator.  The first representation is based on a splitting procedure that combines a renormalized nonlinearity with the linear heat semigroup. The second adapts the nonlinearity to the Caffarelli-Silvestre linear extension technique. The third one is the corresponding nonlinear version of the Balakrishnan formula. We also discuss the correct choice of the constant of the fractional $p$-Laplace operator in order to have continuous dependence as $p\to 2$ and $s \to 0^+, 1^-$.

A number of consequences and proposals are derived. Thus, we propose a natural spectral-type operator in domains, different from the standard restriction of the fractional $p$-Laplace operator acting on the whole space. We also propose  numerical schemes, a new definition of the fractional  $p$-Laplacian on manifolds, as well as alternative characterizations of the $W^{s,p}(\R^n)$ seminorms.

 \medskip


 \medskip


 \end{abstract}

	\section{Introduction}\label{sec:intro}

In  recent years, nonlocal differential operators and equations have attracted an increasing amount of attention. In particular, the fractional Laplacian is defined for suitable functions $u(x)$, $x\in \Rd$, as
\begin{equation}\label{eq.frlap}
	(-\Delta)^s u(x) = c_{n,s}  \PV \int_{\mathbb R^d} \frac{u(x) - u(y)}{|x-y|^{n+2s}} \diff y,
\end{equation}
where $s\in(0,1)$, $c_{n,s}>0$ is a normalization constant and $\PV$ denotes that the integral is taken in the principal value sense, cf. the classical references \cite{Landkof66, Stein1970}.
It appears (but for the constant factor) as the subdifferential of the fractional energy
\begin{equation}\label{Js2}
	{\mathcal J}_{2,s}(u)= \int_{{\mathbb R^d}}\int_{{\mathbb R^d} } \frac{|u(x)-u(y)|^2}{|x-y|^{n+2s}}\,dxdy\,,
\end{equation}
that is related to the definition of the standard norm of the space $W^{s,2}(\Rd)$.
Besides, there are several known representation formulas that have proven to be quite convenient in tackling problems involving fractional operators like \eqref{eq.frlap}, see the list in \cite{Kwa17} where a rather thorough discussion is made. Three of those representations are quite relevant and will be re-examined here in a nonlinear context. They are the so called semigroup (or Bochner's subordination) formula; the Caffarelli-Silvestre extension problem, that allows to write the nonlocal operator as in terms of a local problem in one extra dimension; and Balakrishnan's formula.

\noindent \ (i)	A first representation was given by S. Bochner (see \cite{Bochner1949}). In modern form, it says that
for a smooth enough function $u$ the following formula holds
\begin{equation}\label{eq:BochnerFor}
	(-\Delta)^s u(x)= \frac{1}{\Gamma(-s)} \int_{0}^\infty (e^{t \Delta}[u](x)-u(x)) \frac{\dd t}{t^{1+s}},
\end{equation}
where $e^{t \Delta}$ denotes the heat semigroup, i.\,e., we take $e^{t \Delta}[u](x):=w(x,t)$, where $w$ is the solution of the problem
\begin{equation*}
	\begin{dcases}
		\partial_t w(x,t)- \Delta w(x,t)=0, & x \in \Rd , t > 0 , \\
		w (x,0) = u(x), & x \in \Rd .
	\end{dcases}
\end{equation*}
\noindent \ (ii)		On the other hand, the fractional Laplacian can be written in local terms since L. Caffarelli and L. Silvestre showed in \cite{Caffarelli2007} that the extension problem
\begin{equation*}
	\begin{dcases}
		\Delta_x v + \frac{1-2s}{y} v_y + v_{yy} = 0 & x \in \Rd , y > 0 , \\
		v (x,0) = u(x) & x \in \Rd \,,
	\end{dcases}
\end{equation*}
allows to compute the fractional Laplacian $(-\Delta)^s u$ as a weighted Neumann limit of $v$
\begin{equation}
	\label{eq:extension}
	(-\Delta)^s u (x) = \frac{4^s \Gamma(s)}{\Gamma(-s)} \lim_{y \to 0} \frac{v(x,y) - v(x,0)}{y^{2s}}.
\end{equation}

\noindent \ (iii) A further representation for the fractional Laplacian was given by A. V. Balakrishnan whose formula  reads (see \cite{Balakrishnan1960})
\begin{equation}
	\label{eq:Balakrishnan original}
	(-\Delta)^s u(x)= \frac{\sin(\pi s)}{\pi} \int_0^\infty (-\Delta)(t-\Delta)^{-1} [u](x) \frac{\dd t}{t^{1-s}}.
\end{equation}

The three representations have proved to be very useful both in theory and applications, and they have given rise to a wide literature.

\medskip

\noindent {\bf Aim and main results.}	The aim of this paper is to present analogous representations in the nonlinear context of the fractional $p$-Laplacian, and we comment on their usefulness in different applications. We first recall that the fractional $p$-Laplacian is usually defined as follows
\begin{equation}\label{eq.spl1}
	\fpl u(x) =C_1  \PV \int_{\mathbb R^d} \frac{\Phi_p (u(x) - u(y))}{|x-y|^{n+sp}} \diff y,
\end{equation}
where $p\in (1,\infty)$ and $s\in(0,1)$, we write $\Phi_p (t) = |t|^{p-2}t$, and $C_1=C_1(n,s,p)>0$
is a constant. The fractional $p$-Laplacian is the subdifferential of a convex functional called the $(s,p)$-energy (or Gagliardo seminorm)
\begin{equation}\label{Jsp1}
	{\mathcal J}_{p,s}(u)= \frac{ C_1}{p}\int_{{\mathbb R^d}}\int_{{\mathbb R^d} } \frac{|u(x)-u(y)|^p}{|x-y|^{n+sp}}\,dxdy\,,
\end{equation}
defined for those functions in $L^2(\Rd)$  with finite $(s,p)$-energy, see e.g. \cite{Mazon2016, Vazquez2020}. We also recall that this operator has a dense domain in $L^2(\Rd)$ containing smooth enough functions. Actually, $\fpl u(x)$  is well-defined in the  pointwise sense for smooth functions $C_c^\infty (\Rd)$ (assuming also that the gradient does not vanish when $p$ lies in the range $p\in(1,\frac{2}{2-s}$) (see, e.g., Section 3 in \cite{Korvenpaa2019}).

Our main result is the introduction of three equivalent representation formulas for $\fpl$, which retain the flavour of the linear case. First, we show that 	the fractional $p$-Laplacian can be represented by the formula
\begin{equation}\label{eq:semigintro}
	\fpl u(x) = 	C_2
	\int_0^{+\infty} e^{t \Delta} [\Phi_p (  u(x) - u(\cdot) )](x) \frac{\diff t}{t^{1+\frac {sp} 2}}\,,
\end{equation}
where $C_2 > 0$.
The precise study of this nonlinear representation occupies \Cref{sec:semigrep},  see \Cref{thm:semigroup}.

Second, we consider an extension operator $E_{s,p}$ related to a local problem in $ \R^n \times (0,\infty)$, and such that the following representation holds:
\begin{equation}\label{eq:extintro}
	\fpl u(x) =  C_3\,   \lim_{y \to 0} \frac{ E_{s,p}\Big [\Phi_p \big (u(x)-u(\cdot) \big ) \Big ](x,y)  }{y^{sp}},
\end{equation}
where $C_3 > 0$.
The precise study of this nonlinear representation occupies \Cref{sec:extension}, see \Cref{thm:extension}.

Lastly, the fractional $p$-Laplacian is also equivalently given by the following nonlinear Balakrishnan's formula
\begin{equation}\label{eq:Balaintro}
	\fpl u(x) =  C_4 \int_0^\infty \Delta (t-\Delta)^{-1} [\Phi_p (  u(x) - u(\cdot) )](x) \frac{\diff t}{t^{1-\frac {sp} 2}},
\end{equation}
where $C_4 > 0$. 	This will be studied in \Cref{sec:Balak}, see \Cref{thm:Balak}.

The choice of the  constant $C_1$ is indifferent in most cases and can be fixed as 1, but a precise value $C_1(n,s,p)$ can be important in writing some exact formulas or in the limit cases $s\to 0^+$, $s\to 1^-$, or $p\to 2$.  In \Cref{sec:constants} we will discuss the choice of this constant and also give an explicit relation with $C_2$, $C_3$, and $C_4$, which turn out to be dimensionally free after a suitable choice of $C_1$.

The three above representations succeed in replicating the linear formulas in the nonlinear setting, which was unexpected. One may wonder about the complexity introduced by the nonlinear term $\Phi_p(u(x) - u(y))$. In that sense, we have the following comments.
Firstly, we note that the operator used in each case $(e^{t\Delta}$, $E_{s,p}$ and $\Delta (t-\Delta)^{-1}$) is independent of the point $x$ where  $\fpl u(x)$ is computed.
As a consequence of that, our representation formulas  need to solve a single problem with an initial/boundary condition depending on a parameter, i.e.  $f_\lambda(y)=\Phi_p(\lambda - u(y))$ where $\lambda =u(x)$. This fact is very clearly seen in the numerical application, since formula \eqref{eq:discFPL2} does not add any extra complexity with respect its linear counterpart in the computation of the coefficients. Another possible interpretation of our representations is that given a function $u:\R^n\to \R$  we build an initial/boundary condition $F(x,z):=\Phi_p(u(x)-u(z))$ and solve a problem in $\R^{2n+1}_+$ where the operator only acts in $n+1$ of the variables. We continue next with some further consequences.

\noindent{\bf Representation in bounded domains.}	As a first consequence of the above results, we remark that the first representation of the fractional $p$-Laplacian  naturally suggests a version of the operator defined in bounded domains $\Omega$ of $\R^n$. Given $u:\Omega\to \R$ we define a new operator
\begin{equation}\label{eq:semigdomain}
	(-\Delta_\Omega)_{p}^s u (x) :=C_2\int_0^{+\infty} e^{t \Delta_\Omega} [\Phi_p (  u(x) - u(\cdot) )](x) \frac{\diff t}{t^{1+\frac {sp} 2}},
\end{equation}
where $e^{t \Delta_\Omega} $ denotes the heat semigroup in $\Omega$ with Dirichlet boundary conditions, i.e. $w(x,t):=e^{t \Delta_\Omega}[f](x)$ is the solution of
\begin{equation*}
	\begin{dcases}
		w_t(x,t)- \Delta w(x,t)=0 & x\in\Omega, t > 0, \\
		w(x,t) = 0 & x \in \partial \Omega, t > 0 , \\
		w(x,0) = f (x) & x \in \Omega.
	\end{dcases}
\end{equation*}
This operator is defined for any $p>1$ and generalises the Spectral Fractional Laplacian, which is the well-known case $p = 2$.
In \Cref{sec:new spectral} we will show this operator is well-defined and different from the restriction of $\fpl$. Furthermore, we suggest a generalisation for manifolds that we do not develop in full detail. Analogous variants can be constructed via \eqref{eq:extintro} and \eqref{eq:Balaintro}.

\medskip

\noindent{\bf Further extensions and comments.}
\Cref{sec:numApp} addresses the topic of numerical analysis for the fractional $p$-Laplacian starting from our representation formulas.   \Cref{sec:app} is devoted to comment on some extensions and applications of our results. In the \Cref{sec:commentandproblems}, we present some comments and open problems that we find interesting as well as a brief intuition about our results. Finally, we include an appendix with technical results.

\medskip

\noindent{\bf Comments on  related literature.}
Some references for stationary boundary value problems involving the fractional $p$-Laplace operator that we study here:	
Ishii and Nakamura \cite{Ishii10} treat the first boundary value problem;  Lindgren and Lindqvist \cite{LindgrenLindq14, Lindgren17} establish the fractional eigenvalues, Perron's Method and Wiener's theorem. Regularity was studied in \cite{Brasco17,BrLiSc18,brasco2019,Iannizzotto16,Kuusi15,Lindgren16} among others. The equivalence of weak and viscosity solutions in boundary value problems involving the fractional $p$-Laplacian  has been recently treated in \cite{Korvenpaa2019} in the homogeneous case and \cite{barrios2020} in the non-homogeneous case. The parabolic problem has been studied in bounded domains: Mazón, Rossi, and Toledo \cite{Mazon2016}, Puhst \cite{Puhst15}, and  Vázquez \cite{Vazquez2016}; and also in the whole space $\R^n$ by M. Strömqvist \cite{Strom19} and Vázquez \cite{Vazquez2020}.

We will make specific comments on the semigroup formula, the extension problem and the Balakrishnan's formula in the corresponding sections. Other applications, in particular numerical studies are mentioned in \Cref{sec:numApp} and  \Cref{sec:app}.

The literature contains works on other types of operators that are also known as  ``nonlocal or fractional $p$-Laplacian  operators''.  We are not treating such versions in our paper since they lead to quite different objects, tools and results. For instance, in the book by  Andreu-Vaillo, Mazón, Rossi, Toledo, \cite{Andreu2010}, the integro-differential operator is defined on the basis of integrable kernels.  Cipriani and Grillo proposed in \cite{CiGr09} another type of nonlinear and nonlocal operator defined via the subordination formula applied directly to a nonlinear semigroup. We show in the last section that this method leads to an operator that is different from $\fpl$ treated here. On the other hand, Bjorland-Caffarelli-Figalli introduce in  \cite{BCF12a, BCF12b} a kind of normalized  (game theoretical)  fractional infinity Laplacian as well as a fractional $p$-Laplace operator in non-divergence form (see \cite{Manfredi10,Manfedi12} for the local version).  They are the nonlocal counterpart of the normalized standard $p$-Laplacian in the viscosity sense.
Finally,  Chasseigne and Jakobsen \cite{ChaJa17} introduced a related non-normalized nonlocal $p$-Laplacian in the viscosity sense.

As a final comment, we recall that  the standard $p$-Laplace operator  ($s=1$) has been widely studied both in the stationary and evolution equations. Let us mention some classical  monographs by  DiBenedetto \cite{DiBenedetto93},  Lindqvist \cite{Lindqvist06}, that contain further references. The monograph \cite[Section 11]{Vazquez2006} contains a summary of the evolution results.

\section{Semigroup representation}\label{sec:semigrep}
We present an interesting representation of  the operator using the heat semigroup.
Let us use the notation $e^{t \Delta}$  for this  semigroup, i.\,e., given $f:\R^n\to \R$ we denote by $w(x,t)=e^{t \Delta}[f](x)$ the solution of
\begin{equation*}
	\begin{dcases}
		\partial_t w(x,t)- \Delta w(x,t)=0, & x \in \Rd , t > 0 , \\
		w (x,0) = f(x), & x \in \Rd .
	\end{dcases}
\end{equation*}
The following is our first main result.
\begin{theorem}
	\label{thm:semigroup}
	Let $s \in (0,1)$ and $p \in (1,+\infty)$. Assume $u\in C^2_b(\Rd)$ and $x_0\in \Rd$. If $p \in (1,\frac{2}{2-s})$ assume additionally that $\nabla u(x_0)\not=0$. Then we have the following representation of the fractional $p$-Laplacian
	\begin{equation}\label{repr1}
		\fpl u(x_0) = 	C_2 (n,s,p)
		\int_0^{+\infty} e^{t \Delta} [\Phi_p (  u(x_0) - u(\cdot) )](x_0) \frac{\diff t}{t^{1+\frac {sp} 2}}, \qquad
	\end{equation}
	where $ C_2 (n,s,p) = \frac {\pi^{ \frac n 2 } C_1(n,s,p)}  {2^{sp} \Gamma(\frac{n+sp}2)} $.
\end{theorem}
Note that, when $p=2$, we have that
\[
e^{t \Delta} [\Phi_p (  u(x) - u(\cdot) )](x)=e^{t \Delta} [  u(x) - u(\cdot) ](x)= u(x)- e^{t \Delta} [u](x).
\]
so that, a suitable choice of $	C_1 (n,s,p)$ allows to recover Bochner's formula \eqref{eq:BochnerFor} from \eqref{repr1}.  Notice that $\Gamma(-s) < 0$ while $C_2 > 0$.

\bigskip

Our proof is based on the scaling properties of the representation formula
\begin{equation*}
	e^{t \Delta} [f] (x) =\frac{1}{(4 \pi t)^{ \frac n 2 }}  \int_{\Rd} {e^{-\frac{|x-y|^2}{4t}}} f(y) \diff y.
\end{equation*}
We use here it in combination with the nonlinearity, plus a desingularization step that is done by subtraction inside the argument of the nonlinearity. In this way we obtain the {\sl splitting representation} \eqref{repr1} of the fractional $p$-Laplacian.

\begin{proof}[Proof of \Cref{thm:semigroup}]
	Let $\veps>0$ and consider the bounded function $g_\veps:\R^n\times \R^n\to 0$ given by $g_\veps(x,y)= \Phi_p ( u(x) - u(y) ) \chi_{|x-y| > \varepsilon}$.
	Fix $x\in \R^n$ and consider the heat semigroup associated to $g_\veps(x,\cdot)$, i.e.
	\begin{equation*}
		e^{t \Delta} [g_\veps(x, \cdot)] (z) = \frac{1}{(4 \pi t)^{ \frac n 2 }}  \int_{\Rd} {e^{-\frac{|z-y|^2}{4t}}} g_\veps(x,y) \diff y.
	\end{equation*}
	Taking $z = x$ we have
	\begin{equation*}
		e^{t \Delta} [g_\veps(x, \cdot)] (x) = \frac{1}{(4 \pi t)^{ \frac n 2 }}  \int_{\Rd} {e^{-\frac{|x-y|^2}{4t}}}g_\veps(x, y) \diff y.
	\end{equation*}
	Let now $\delta>0$ and integrate the quantity above with respect the singular measure $\dd t/t^{1+\frac{sp}{2}}$ in the truncated domain $(\delta, +\infty)$ (we do this to avoid integrability issues around the origin due to the singular behaviour of the measure) to obtain
	\begin{align*}
		\int_\delta^{\infty} e^{t \Delta} [g_\veps(x,\cdot)] (x) \frac{\diff t}{t^{1+\frac {sp} 2}} &= \int_\delta^{\infty}  \frac{1}{(4 \pi t)^{ \frac n 2 }} \left( \int_{\Rd} {e^{-\frac{|x-y|^2}{4t}}} g_\veps(x,y) \diff y\right) \frac{\diff t}{t^{1+\frac {sp} 2}}\\
		&=   \int_{\Rd}  \left( \int_\delta^{\infty}  \frac{1}{(4 \pi t)^{ \frac n 2 }} {e^{-\frac{|x-y|^2}{4t}}}\frac{\diff t}{t^{1+\frac {sp} 2}}\right) g_\veps(x,y) \diff y
	\end{align*}
	where we have used Fubini's theorem to change the order of integration. We use now the change of variables $t = |x-y|^2 /(4\tau)$ to get
	\begin{align*}
		\int_\delta^{\infty} e^{t \Delta} [g_\veps(x,\cdot)]& (x) \frac{\diff t}{t^{1+\frac {sp} 2}}
		\\
		&=\frac {2^{sp}} {\pi^{ \frac n 2 }}   \int_{\Rd} \left( \int_{0}^{ \frac{|x-y|^2}{4\delta}}  {e^{-\tau}}  \frac{\diff \tau}{\tau^{  1 - \frac{n+sp}2 }}  \right)   \frac{ g_\veps(x,y) } { |x-y|^{ n+sp}}  \diff y.
	\end{align*}
	Note that, since $g_\veps(x,y)=0$ if $|x-y|<\veps$, we can consider the above expression to be integrated in $\R^n \setminus B_\veps(x)$. Now, we want to let $\delta \to 0^+$ to get
	
	\begin{align*}
		\int_0^{\infty} e^{t \Delta} [g_\veps(x,\cdot)] (x) \frac{\diff t}{t^{1+\frac {sp} 2}}  &= \frac {2^{sp}} {\pi^{ \frac n 2 }}   \left( \int_{0}^{\infty}  {e^{-\tau}} \tau^{ - 1 + \frac{n+sp}2 } \diff \tau \right)   \int_{\Rd}   \frac{ g_\veps(x,y) } { |x-y|^{ n+sp}}  \diff y \\
		&=\frac {2^{sp} \Gamma(\frac{n+sp}2)} {\pi^{ \frac n 2 }}      \int_{\Rd}   \frac{ g_\veps(x,y) } { |x-y|^{ n+sp}}  \diff y.
	\end{align*}
	
	\noindent\textbf{Step 1:} Let either $p\in [2,\infty)$ or $p\in(1,2)$ and $\nabla u(x)\not=0$.
	\\
	Passing to the limit in the right hand side is trivial since the integrand converges pointwise and monotonically to ${e^{-\tau}} \tau^{ - 1 + \frac{n+sp}2}$ which is an integrable function in the domain $(0,+\infty)$, which yields to the definition of the $\Gamma$ function. To pass to the limit in the left hand side,  we estimate
	
	\begin{align*}
		\left|e^{t \Delta} [g_\veps(x,\cdot)](x)\right|&=   \frac{1}{(4 \pi t)^{ \frac n 2 }}\left|  \int_{|x-y|>\veps} {e^{-\frac{|x-y|^2}{4t}}} \Phi_p ( u(x) - u(y) ) \diff y\right|\\
		&\lesssim  \frac{1}{ t^{ \frac n 2 }} \int_{\R^n} {e^{-\frac{|z|^2}{4t}}} |z|^p \diff z\\
		&=  \frac{1}{ t^{ \frac n 2 }}  \int_{\R^n} {e^{-\frac{r^2}{4}}}   t^{\frac{p}{2}} |r|^p t^{\frac{n}{2}} \diff r= t^{\frac{p}{2}}  \int_{\R^n} {e^{-\frac{r^2}{4}}} |r|^p\diff r \lesssim t^{\frac{p}{2}},
	\end{align*}
	where we have used either \Cref{lem:tech1} or \Cref{lem:tech3}  with $$K(z)=  e^{-\frac{|z|^2}{4t}} \chi_{|z| > \veps}(z)$$ and the change of variables $z=t^{\frac{1}{2}}r$ i.e. $\dd z= t^{\frac{n}{2}}\dd r$.
	Then
	\[
	\begin{aligned}
	\int_\delta^{1} |e^{t \Delta} [g_\veps(x,\cdot)] (x) | \frac{\diff t}{t^{1+\frac {sp} 2}}
	&\lesssim  \int_\delta^{1} t^{\frac{p}{2}} \frac{\diff t}{t^{1+\frac {sp} 2}} = \int_\delta^{1} \frac{\diff t}{t^{1-\frac {(1-s)p} 2}} \\
	& \leq \int_0^{1} \frac{\diff t}{t^{1-\frac {(1-s)p} 2}}  <+\infty
	\end{aligned}
	\]
	where the bound comes from the fact that $(1-s)p/2>0$ since $s\in(0,1)$.

	Now we want to take limits as $\veps\to0$. In the right hand side, we use either \Cref{lem:tech1} or \Cref{lem:tech3}  with $K(z)= |z|^{-n-sp} \chi_{|z| < \veps}(z)$ to get
	\begin{align*}
		&\Bigg|\textup{P.V.}  \int_{\Rd}   \frac{ \Phi_p(u(x)-u(y)) } { |x-y|^{ n+sp}}  \diff y- \textup{P.V.}  \int_{\Rd}   \frac{ g_\veps(x,y) } { |x-y|^{ n+sp}}  \diff y\Bigg|\\
		& \leq \Bigg|\textup{P.V.}  \int_{|x-y|<\veps}   \frac{ \Phi_p(u(x)-u(y)) } { |x-y|^{ n+sp}}  \diff y \Bigg|\\
		&\lesssim \int_{|z|<\veps}   \frac{ |z|^p } { |z|^{ n+sp}}  \diff z  \\
		&\lesssim  \veps^{(1-s)p} \to 0 \qquad \textup{as} \qquad \veps\to0.
	\end{align*}
	While in the left hand side we use either \Cref{lem:tech1} or \Cref{lem:tech3}  with $K(z)= e^{-\frac{|z|^2}{4t}} \chi_{|z| < \veps}(z)$ to estimate for $t\in(0,1)$,
	\begin{align*}
		&\Bigg| \left(e^{t \Delta} [\Phi_p (  u(x) - u(\cdot) \right)](x) - e^{t \Delta} [g_\veps(x,\cdot)] (x)  \Bigg| \\
		&=   \frac{1}{(4 \pi t)^{ \frac n 2 }} \left|\int_{|x-y|<\veps} {e^{-\frac{|x-y|^2}{4t}}} \Phi_p ( u(x) - u(y) ) \dd y\right|\\
		&\lesssim   \frac{1}{ t^{ \frac n 2 }} \int_{|z|<\veps} {e^{-\frac{|z|^2}{4t}}} |z|^p\dd z
		\lesssim t^\frac{p}{2} \int_{|r|< \veps t^{-\frac{1}{2}}} {e^{-\frac{r^2}{4}}} |r|^p\dd r
	\end{align*}
	and for $t\geq1$,
	\[\begin{split}
		&| e^{t \Delta} [\Phi_p (  u(x) - u(\cdot) )](x) - e^{t \Delta} [g_\veps(x,\cdot)] (x)| \\
		&\qquad \leq \sup_{|x-y|<\veps} |\Phi_p (  u(x) - u(y) )| \\
		&\qquad \lesssim \sup_{|x-y|<\veps} |u(x)-u(y)|^{\min\{1,p-1\}} \leq \veps^{\min\{1,p-1\}}
	\end{split}
	\]
	Then
	\begin{equation}
		\label{eq:convergence semigroup}
		\begin{aligned}
			&\Bigg|\int_0^{+\infty}e^{t \Delta} [\Phi_p (  u(x) - u(\cdot) )](x) \frac{\diff t}{t^{1+\frac {sp} 2}}- \int_0^{\infty} e^{t \Delta} [g_\veps(x,\cdot)] (x) \frac{\diff t}{t^{1+\frac {sp} 2}}\Bigg|\\
			&\leq \int_0^1 \left(\int_{|r|< \veps t^{-\frac{1}{2}}} {e^{-\frac{r^2}{4}}} |r|^p\dd r \right) \frac{\diff t}{t^{1-\frac {(1-s)p} 2}} +  \veps^{\min\{1,p-1\}} \int_1^\infty \frac{\diff t}{t^{1+\frac {sp} 2}}.
		\end{aligned}
	\end{equation}
	Clearly, the second term in the last estimate goes to zero as $\veps\to0$. For the first term, let
	\[
	f_\veps(t):=\int_{|r|< \veps t^{-\frac{1}{2}}} {e^{-\frac{r^2}{4}}} |r|^p\dd r.
	\]
	For the first term, clearly $f_\veps \to0$ a.e in $t\in[0,\infty)$,
	\[
	|f_\veps(t)| \leq F(t) \equiv  \int_{\Rd} {e^{-\frac{r^2}{4}}} |r|^p\dd r
	\]
	and
	\[
	\int_0^1 F(t)  \frac{\diff t}{t^{1-\frac {(1-s)p} 2}} =  \int_{\Rd} {e^{-\frac{r^2}{4}}} |r|^p\dd r \int_0^1 \frac{\diff t}{t^{1-\frac {(1-s)p} 2}} <+ \infty,
	\]
	so by the Dominated Convergence Theorem this term also converges to zero as $\veps\to0$.
	
	\medskip
	\noindent \textbf{Step 2: } Now let $p\in(\frac{2}{2-s},2)$.  We can use  \Cref{lem:tech2} to get
	\begin{align*}
		\left|e^{t \Delta} [g_\veps(x,\cdot)](x)\right|& \lesssim t^{p-1}.
	\end{align*}
	so that
	\[
	\int_\delta^{1} |e^{t \Delta} [g_\veps(x,\cdot)] (x) | \frac{\diff t}{t^{1+\frac {sp} 2}} \lesssim \int_0^{1} t^{p-1}\frac{\diff t}{t^{1+\frac {sp} 2}}  <+\infty
	\]
	which is finite since $p-1>\frac{sp}{2}$ since $p>\frac{2}{2-s}$. That allow us to pass to the limit as $\delta \to0$. The limit as  $\veps\to0$ follows also as in Step 1 one using  \Cref{lem:tech2} instead of  \Cref{lem:tech1}.
\end{proof}

\section{The new extension problem}\label{sec:extension}
The aim of this section is to write  corresponding  formulas in the present nonlinear setting following the ideas introduced in the linear setting by Caffarelli and Silvestre \cite{Caffarelli2007}.

For $s \in (0,1)$ and $p \in (1,+\infty)$ consider the extension problem
\begin{equation}
	\label{eq:extension2}
	\begin{dcases}
		\Delta_x U(x,y) + \frac{1-sp}{y} U_y(x,y) + U_{yy}(x,y) = 0 & x \in \Rd , y > 0 , \\
		U (x,0) = f(x)& x \in \Rd ,
	\end{dcases}
\end{equation}
which extends the usual theory where $p = 2$. We will show that a classical solution of this problem can be recovered by convolution with an explicit Poisson kernel
\begin{equation}\label{eq:ext-form2}
	\begin{aligned}
	U(x,y) &= \int_{ \Rd } P(x-\xi, y)  f(\xi)  \diff \xi\,,
	\\ 	
	&\text{where} 	\quad
	P(x,y) =  \frac{\Gamma ( \frac {n + sp}2 )}{ \pi^{\frac n 2} \Gamma (\frac{sp}2)} \frac{y^{sp} }{(|x|^2 + y^2)^{\frac {n+sp}2}}.
	\end{aligned}
\end{equation}
Finally, let us define the extension operator $E_{s,p} [f] := U$ through  formula \eqref{eq:ext-form2}. With this extension operator we can give a representation formula for the $(s,p)$-Laplacian.

\begin{theorem}\label{thm:extension}
	Let $s \in (0,1)$ and $p \in (1,+\infty)$.
	Assume $u\in C^2_b(\Rd)$ and $x_0\in \Rd$. If $p \in (1,\frac{2}{2-s})$ assume additionally that $\nabla u(x_0)\not=0$. We have the following representation of the fractional $p$-Laplacian:
	\begin{equation}
		\label{eq:representation extension}
		\fpl u(x_0) =  C_3(n,s,p)\,   \lim_{y \to 0} \frac{ E_{s,p}\Big [\Phi_p \big (u(x_0)-u(\cdot) \big ) \Big ](x_0,y)  }{y^{sp}}
	\end{equation}
	where the constant is given by $C_3 (n,s,p) =  \frac{ \pi^{\frac n 2} \Gamma (\frac{sp}2)} {\Gamma ( \frac {n + sp}2 )} C_1(n,s,p) $.
\end{theorem}

Actually, for $p=2$ we have that
\[
\begin{aligned}
\frac{ E_{s,p}[\Phi_p(u(x)-u(\cdot))](x,y)  }{y^{sp}}&=\frac{ E_{s,2}[u(x)-u(\cdot)](x,y)  }{y^{2s}}\\
&= \frac{ E_{s,2}[u](x,0)-E_{s,2}[u](x,y)  }{y^{2s}}
\end{aligned}
\]
recovering the classical result of \cite{Caffarelli2007} for the fractional Laplacian.

\begin{remark}
	Notice that, by L'Hôpital's rule
	\begin{equation*}
		\fpl u(x_0) = \frac{C_3(n,s,p)}{sp}  \lim_{y \to 0}  y^{1-sp} \partial_y \left(E_{s,p}[\Phi_p(u(x_0)-u(\cdot))] \right)(x_0,y).
	\end{equation*}
\end{remark}
In order to prove \Cref{thm:extension} we need an important lemma that gives an alternative representation of the extension operator based on the one by Stinga and Torrea (\cite{Stinga2010}) for fractional powers of general linear operators.

\begin{lemma}\label{lem:extension}
	Let $s \in (0,1)$ and $p \in (1,+\infty)$. If  $f\in C_b(\Rd)$ then
	$E_{s,p}[f]$ is a classical  solution of \eqref{eq:extension2}. Furthermore, it can be alternatively  represented by
	\begin{equation}\label{eq:ext-form1}
		E_{s,p}[f](x,y)=  \frac{y^{sp}}{2^{sp}\Gamma(\frac {sp}2)} \int_0^\infty  e^{t\Delta} [f](x)  \,  e^{-\frac {y^2}{4t} } \frac{dt}{t^{1+\frac{sp}2}}.
	\end{equation}
\end{lemma}

\begin{proof} We proceed in several steps.
	
	\noindent \textbf{Step 1.} We will check here that the pointwise formulas \eqref{eq:ext-form1} and \eqref{eq:ext-form2} are well defined for $y>0$.
	
	First, for \eqref{eq:ext-form1}, we use the explicit representation of $e^{t\Delta}$, and the fact that $f$ is bounded
	\begin{align*}
		&\left|\int_0^\infty e^{t\Delta}\left[f \right] (x) \,  e^{-\frac {y^2}{4t} } \frac{dt}{t^{1+\frac{sp}2}}\right|\\
		&= \left|\int_0^\infty e^{-\frac {y^2}{4t} } \left(\int_{\R^n} \frac{1}{(4 \pi t)^{ \frac n 2 }} e^{-\frac{|x-\xi|^2}{4t}}f(\xi) \diff \xi  \right)  \frac{dt}{t^{1+\frac{sp}2}}\right|\\
		& \lesssim \int_0^\infty e^{-\frac {y^2}{4t} } \left(\int_{\R^n} \frac{1}{(4 \pi t)^{ \frac n 2 }} e^{-\frac{|\xi|^2}{4t}} \diff \xi  \right)  \frac{dt}{t^{1+\frac{sp}2}} \leq  \int_0^\infty e^{-\frac {y^2}{4t} }   \frac{dt}{t^{1+\frac{sp}2}}.
	\end{align*}
	The last term is clearly finite since $sp/2>0$ and  $e^{-\frac {y^2}{4t} } /t^{\alpha} \to 0$ as $t\to 0$ for any $\alpha\in \R$ as long as $y>0$.
	
	Second, for  \eqref{eq:ext-form2},
	\[
	\left|\int_{ \Rd }  \frac{f(\xi) }{(|x-\xi|^2 + y^2)^{\frac {n+sp}2}} \diff \xi\right |\lesssim \int_{ \Rd }  \frac{\dd \xi  }{(|\xi|^2 + y^2)^{\frac {n+sp}2}}
	\]
	which again is finite since $K(\xi):=(|\xi|^2+y^2)^{-\frac{n+2s}{2}}$ is bounded as long as $y>0$ and $K(\xi)\leq |\xi|^{-(n+2s)}$ which is an integrable function in $\R^n\setminus B_1$.
	
	\textbf{Step 2.}  Now we check that both pointwise formulas are equivalent. Notice that, by Step 1, all the integrals below are well defined, and they allow us to do Fubini to change the order of integration. Then
	
	\begin{align*}
		&\frac{y^{sp}}{2^{sp}\Gamma(\frac {sp}2)} \int_0^\infty  e^{t\Delta} \left[f \right] (x) \,  e^{-\frac {y^2}{4t} } \frac{dt}{t^{1+\frac{sp}2}}\\
		& = \frac{y^{sp}}{2^{sp}\Gamma(\frac {sp}2)} \int_0^\infty \int_{\Rd} \frac{1}{(4 \pi t)^{ \frac n 2 }} e^{-\frac{|x-\xi|^2+y^2}{4t}}f(\xi)\diff \xi \frac{dt}{t^{1+\frac{sp}2}}  \\
		&=  \int_{\Rd} \left( \frac{y^{sp}}{2^{n+sp} \pi^{\frac n 2}\Gamma(\frac {sp}2)} \int_0^\infty  e^{-\frac{|x-\xi|^2+y^2}{4t}}\frac{dt}{t^{1+\frac{n+sp}2}}   \right) f(\xi) \diff \xi\\
		&= \int_{\Rd}  \left(\frac{ y^{sp}}{ \pi^{\frac n 2}\Gamma(\frac {sp}2) (|x-\xi|^2 + y^2)^{\frac{n+sp}2}}   \int_{0}^{\infty}  {e^{-\tau}} \tau^{ - 1 + \frac{n+sp}2 } \diff \tau\right) f(\xi) \dd \xi\\
		&= \int_{\Rd} P(x-\xi,y) f(\xi) \dd \xi.
	\end{align*}
	
	\textbf{Step 3. } Now we check that formulas \eqref{eq:ext-form1} and \eqref{eq:ext-form2} satisfy  \eqref{eq:extension2}. We start with the boundary condition. To do so, it is sufficient to show that 	
	\begin{equation*}
		P(x,y) \to \delta_0 (x) \quad \textup{as} \quad y\to0^+.
	\end{equation*}
	For $x\not=0$ it is easy to check from \eqref{eq:ext-form2} that we have
	\[
	P(x,y) \to 0  \quad \textup{as} \quad y\to0^+.
	\]
	It is also standard to get
	\begin{equation*}
		\begin{split}
			\int_{\Rd}P(x,y)\dd x&=  \frac{y^{sp}}{2^{sp} \Gamma(\frac {sp}2)}  \int_0^\infty e^{\frac{y^2}{4t}} \left( \int_{\Rd}\frac{1}{(4\pi t)^\frac{n}{2}} e^{-\frac{|x|^2}{4t}} \dd x \right) \frac{dt}{t^{1+\frac{sp}2}}\\
			&
			=  \frac{y^{sp}}{2^{sp} \Gamma(\frac {sp}2)}  \int_0^\infty e^{\frac{y^2}{4t}}  \frac{dt}{t^{1+\frac{sp}2}}=1.
		\end{split}
	\end{equation*}
	Finally, it is easy to compute (see for example \cite{Caffarelli2007}) that for all $y>0$ we have
	\begin{align*}
		\Delta_x P + \frac{1 - sp} y \frac{\partial P}{\partial y} + \frac{\partial^2 P}{\partial y^2} = 0.
	\end{align*}
	We have that $U$ is a pointwise solution of the extension equation.
\end{proof}

\begin{proof}[Proof of \Cref{thm:extension}]
	
	We need to check \eqref{eq:representation extension}. For $y>0$ we have, from \Cref{lem:extension}, the following identity
	\begin{align*}
		&C_3\frac{  E_{s,p}\Big [\Phi_p \big (u(x_0)-u(\cdot) \big ) \Big ](x_0,y)  }{y^{sp}} \\
		 &= C_3 \int_{ \Rd } \frac{ P(x_0-\xi, y) }{y^{sp}} \Phi_p (u(x_0) - u(\xi))  \diff \xi.\\
		&=C_1 \int_{ \Rd } \frac{ \Phi_p (u(x_0) - u(\xi)) }{(|x_0-\xi|^2+y^2)^{\frac{n+sp}{2}}} \diff \xi.
	\end{align*}
	Then,
	\begin{align*}
		&\Bigg|\fpl u(x) -C_3\frac{  E_{s,p}\Big [\Phi_p \big (u(x_0)-u(\cdot) \big ) \Big ](x_0,y)  }{y^{sp}}  \Bigg|\\
		&= C_1 \left| \int_{ \Rd }  \Phi_p (u(x_0) - u(\xi)) \left(\frac{1}{|x_0-\xi|^{n+sp}}- \frac{1}{(|x_0-\xi|^2+y^2)^{\frac{n+sp}{2}}}\right) \diff \xi \right|.
	\end{align*}
	Let us denote
	\[
	K_y(\eta)= \left(\frac{1}{|\eta|^{n+sp}}- \frac{1}{(|\eta|^2+y^2)^{\frac{n+sp}{2}}}\right)
	\]
	and estimate
	\begin{align*}
		&\Bigg|\fpl u(x) -C_3\frac{  E_{s,p}\Big [\Phi_p \big (u(x_0)-u(\cdot) \big ) \Big ](x_0,y)  }{y^{sp}}  \Bigg|\\
		&\qquad \qquad \lesssim \left| \int_{ |x_0-\xi|\leq1 } \Phi_p (u(x_0) - u(\xi)) K_y(x_0-\xi) \diff \xi \right| \\
		&\qquad \qquad   + \left| \int_{ |x_0-\xi|>1 } K_y(x_0-\xi) \diff \xi \right|.
	\end{align*}
	Note that the last term goes to zero by the Monotone Convergence Theorem since $K_y(x-\xi)\to0$ as $y\to0$ in a monotone way.  Now, from either  \Cref{lem:tech1},  \Cref{lem:tech3} and  \Cref{lem:tech2} we get
	\[
	\left| \int_{ |\eta|\leq1 } \Phi_p (u(x_0) - u(\xi)) K_y(x_0-\xi) \diff \xi \right| \leq \int_{ |\xi|\leq1 } K_y(\xi)|\xi|^\alpha \diff \xi,
	\]
	for $\alpha$ either $p$ or $2p-2$ depending on the range of $p$. This term again goes to zero as $y\to0$ by the Monotone Convergence Theorem.
\end{proof}

\section{Nonlinear Balakrishnan's formula}\label{sec:Balak}

The aim of this section is to prove the corresponding Balakrishnan-type formula for $p \ne 2$.

\begin{theorem}\label{thm:Balak}
	Let $s \in (0,1)$ and $p \in (1,+\infty)$.
	Assume $u\in C^2_b(\Rd)$ and $x_0\in \Rd$. If $p \in (1,\frac{2}{2-s})$ assume additionally that $\nabla u(x_0)\not=0$. We have the following representation of the fractional $p$-Laplacian:
	\begin{equation}\label{eq:BalakSec}
		\fpl u(x_0) =  C_4(n,s,p) \int_0^\infty \Delta (t-\Delta)^{-1} [\Phi_p (  u(x_0) - u(\cdot) )](x_0) \frac{\diff t}{t^{1-\frac {sp} 2}},
	\end{equation}
	where the constant is given by \ $C_4 (n,s,p) =  \frac{ \pi^{\frac n 2} } {2^{sp}\Gamma(\frac{2+sp}{2})\Gamma ( \frac {n + sp}2 )} C_1(n,s,p) $.
\end{theorem}

Notice that, when $p = 2$ we have that
\begin{align*}
	\Delta (t-\Delta)^{-1} [\Phi_p (  u(x_0) - u(\cdot) )] &= \Delta  (t - \Delta)^{-1} [u(x_0) ] - \Delta(t - \Delta)^{-1} [u (\cdot) ] \\
	&= - \Delta(t - \Delta)^{-1} [u].
\end{align*}
Hence, we recover the usual Balakrishnan's formula.

To show \eqref{eq:BalakSec} in \Cref{thm:Balak} we need some preliminary results concerning the representation formulas of $ (t-\Delta)^{-1} $ and  $\Delta (t-\Delta)^{-1} $.

\begin{lemma}\label{lem:representationformBalak} Let $f\in L^1(\R^n)\cap L^\infty(\R^n)$.
	The following representation formulas hold for all $t\geq0$:
	\begin{enumerate}[{\rm (a)}]
		\item\label{repfor-item1} $(t-\Delta)^{-1}f=K_t*f$ with $\displaystyle K_t(x)=\int_0^\infty e^{-t\tau} G(\tau,x) \dd \tau$ where $G$ is the Gaussian Kernel, i.e.,
		\[
		G(\tau,x)=\frac{1}{(4\pi \tau)^{\frac{n}{2}}} e^{-\frac{|x|^2}{4\tau}}
		\]
		\item\label{repfor-item2} $\Delta(t-\Delta)^{-1}f= R_t*f - f $ where $R_t(x)=tK_t(x)$. Moreover, $R_t(x)=t^{\frac{n}{2}}W(t^{\frac{1}{2}}|x|)$ with $W:\R_+\to\R_+$ given by
		\begin{equation}\label{eq:self-similarP}
			W(\rho)=\frac{\rho^{2-n}}{(4\pi)^{\frac{n}{2}}}\int_0^\infty e^{-\rho^2 w} w^{-\frac{n}{2}} e^{-\frac{1}{4w}} \dd w.
		\end{equation}
	\end{enumerate}
\end{lemma}

\begin{proof}
	First, taking the Fourier transform
	\begin{align*}
		\mathcal F [  (t - \Delta) ^{-1} f  ] (\xi) &= (t + |\xi|^2)^{-1} \mathcal F [f] (\xi)= \left( \int_0^\infty e^{-t \tau} e^{-|\xi|^2 \tau} \diff\tau \right) \mathcal F [f] (\xi) \\
		&= \mathcal F \left[  \left( \int_0^\infty e^{-t \tau} \mathcal F^{-1} [e^{-|\cdot|^2 \tau}] \diff\tau \right) \star f \right] (\xi) .
	\end{align*}
	Since $\mathcal F[G(t, \cdot)] (\xi) = e^{-|\xi|^2 t}$ we recover \eqref{repfor-item1}.
	As shown in \cite[Section 3.2]{Martinez2001}
	we have
	\begin{equation*}
		t K_t - \Delta K_t = \delta_0.
	\end{equation*}
	The intuition for \eqref{repfor-item2}  is clear and it comes from the fact that $\partial_t G - \Delta G = \delta_0$. We can formally write
	\begin{align*}
		\Delta K_t &= \int_0^\infty e^{-t \tau} \Delta G \diff \tau = \int_0^\infty e^{-t\tau} \partial_t G \diff \tau \\
		&= e^{-t\tau} G\Big|_0^\infty + t \int_0^\infty e^{-t\tau} G \diff \tau = 0 - \delta_0 + t K_t.
	\end{align*}
	Hence,
	\begin{equation*}
		\Delta (t- \Delta)^{-1} f = \Delta (K_t \star f) = (\Delta K_t ) \star f = (R_t - \delta_0 ) \star f = R_t \star f - f.
	\end{equation*}
	We check \eqref{eq:self-similarP} now. First, we note that $G(\tau,x)=\tau^{-\frac{n}{2}}F(\tau^{-\frac{1}{2}}|x|)$ with $F:\R_+\to\R_+$ given by
	\[
	F(\rho)=\frac{1}{(4\pi)^{\frac{n}{2}}} e^{-\frac{\rho^2}{4}}.
	\]
	Then,
	\begin{align*}
		R_t(x) &=t \int_0^\infty e^{-t\tau} \tau^{-\frac{n}{2}}F(\tau^{-\frac{1}{2}}|x|) \dd \tau\\
		&= t |x|^{2-n} \int_0^\infty e^{-t|x|^2  w } w^{-\frac{n}{2}}F(w^{-\frac{1}{2}})d w \\
		&= t^{\frac{n}{2}} (t^{\frac{1}{2}}|x|)^{2-n} \int_0^\infty e^{-(t^{\frac{1}{2}}|x|)^2 w}  w^{-\frac{n}{2}}F(w^{-\frac{1}{2}})d w.
	\end{align*}
\end{proof}

We are now ready to prove \Cref{thm:Balak}.

\begin{proof}[Proof of  \Cref{thm:Balak}]
	We proceed in several steps.
	
	\noindent \textbf{Step 1:} We will check here that given $f\in L^\infty(\R^n)$,  $\alpha<0$ and a finite $M>0$ we have that
	\[
	\int_0^M |\Delta (t-\Delta)^{-1}f(x)| \frac{\diff t}{t^{1-\alpha}}<+\infty.
	\]
	From \Cref{lem:representationformBalak} we have
	\[
	\begin{aligned}
	&|\Delta (t-\Delta)^{-1}f(x)|= |R_t*f(x)| +|f(x)|\\
	&\qquad \leq(\|R_t\|_{L^1(\R^N)}+1) \|f\|_{L^\infty(\R^N)}=  2\|f\|_{L^\infty(\R^N)}
	\end{aligned}
	\]
	where we have used the fact that $\|G(\tau,|\cdot|)\|_{L^1(\R^n)}=1$, which implies
	\[
	\|R_t\|_{L^1(\R^N)} = \int_{\R^n} \left(t\int_0^\infty e^{-t \tau}G(\tau,|x|) \dd\tau\right)\dd x= t\int_0^\infty e^{-t \tau} \dd\tau=1.
	\]
	From this observation, we trivially get
	\[
	\int_0^M |\Delta (t-\Delta)^{-1}f(x)| \frac{\diff t}{t^{1-\alpha}}< 2\frac{M^\alpha}{\alpha}\|f\|_{L^\infty(\R^N)} <+\infty.
	\]
	\textbf{Step 2:} Now, given $\veps>0$, we consider the bounded function $g_\veps:\R^n\times \R^n\to \R$ given by $g_\veps(x,y)= \Phi_p ( u(x) - u(y) ) \chi_{|x-y| > \varepsilon}$.  Note that $g_\veps(x,x)=0$. Here we will check that the following identity holds
	\begin{equation}
		\label{eq:Balakidentity1}
		\int_0^\infty \Delta (t-\Delta)^{-1} [g_\veps(x,\cdot)](x) \frac{\diff t}{t^{1-\frac {sp} 2}}= C \int_{|x-y|>\veps} \frac{\Phi_p(u(x)-u(y))}{|x-y|^{n+sp}}\dd y.
	\end{equation}
	where $C= \frac{2^{sp}}{\pi^{\frac{n}{2}}}  \Gamma(1+\frac{sp}{2}) \Gamma(\frac{n+sp}{2})>0$.
	
	\smallskip
	
	\textbf{Step 2 (a):} First we check that
	
	\begin{equation}\label{eq:absoluteconvBalak}
		\int_0^M |\Delta (t-\Delta)^{-1} [g_\veps(x,\cdot)](x)| \frac{\diff t}{t^{1-\frac {sp} 2}}<+\infty
	\end{equation}
	uniformly in $M$. In this way, the Dominated convergence theorem we get
	\begin{equation}\label{eq:limitMBalak}
		\begin{aligned}
		&\int_0^M  \Delta (t-\Delta)^{-1} [g_\veps(x,\cdot)](x) \frac{\diff t}{t^{1-\frac {sp} 2}} \\
		&\qquad \stackrel{M\to\infty}{\longrightarrow} \int_0^\infty  \Delta (t-\Delta)^{-1} [g_\veps(x,\cdot)](x) \frac{\diff t}{t^{1-\frac {sp} 2}}<+\infty.
		\end{aligned}
	\end{equation}
	To check \eqref{eq:absoluteconvBalak}, we assume first that $p\geq2$, and use the explicit representation of $\Delta(t-\Delta)^{-1}$ given in \Cref{lem:representationformBalak} together with \Cref{lem:tech1} to get
	\begin{align*}
		&\int_0^M |\Delta (t-\Delta)^{-1} [g_\veps(x,\cdot)](x)| \frac{\diff t}{t^{1-\frac {sp} 2}} \\
		&=\int_0^1 |\Delta (t-\Delta)^{-1} [g_\veps(x,\cdot)](x)| \frac{\diff t}{t^{1-\frac {sp} 2}}\\
		&\qquad +
		\int_1^M \left|\int_{|x-y|>\veps} R_t (x-y) \Phi_p(u(x)-u(y))\dd y \right|\frac{\diff t}{t^{1-\frac {sp} 2}}\\
		&\leq \frac{2}{sp}\|g_\veps(x,\cdot)\|_{L^\infty(\R^n)}+
		C
		\int_1^\infty\left( \int_{|z|>\veps} R_t (z) |z|^p \dd z\right)\frac{\diff t}{t^{1-\frac {sp} 2}}.
	\end{align*}
	We use now the explicit form of $R_t$ to get
	\begin{align*}
		&\int_1^\infty\left( \int_{|z|>\veps} R_t (z) |z|^p \dd z\right)\frac{\diff t}{t^{1-\frac {sp} 2}}\\
		&= \int_1^\infty\left( \int_{|z|>\veps} t^{\frac{n}{2}}W(t^{\frac{1}{2}}|z|) |z|^p \dd z\right)\frac{\diff t}{t^{1-\frac {sp} 2}} \\
		&=\int_1^\infty  \left( \int_{|y|>\veps t^{\frac{1}{2}}} W(|y|) |y|^pt^{-\frac{p}{2}} \dd y\right)\frac{\diff t}{t^{1-\frac {sp} 2}}\\
		&\leq  \left( \int_{|y|>\veps} W(|y|) |y|^p \dd y\right) \left(\int_1^\infty  \frac{\dd t}{ t^{1+\frac {(1-s)p} {2}}} \right).
	\end{align*}
	It is rather standard to check that the last two integrals are finite. Thus, \eqref{eq:limitMBalak} is proved for $p\geq2$. When $p\in(1,2)$ we use \Cref{lem:tech3} and \Cref{lem:tech2} instead and the proof follows in the same way.

	\smallskip
	
	\textbf{Step 2 (b):} We proceed now by direct computations:
	\begin{align}\label{eq:BalaklimM}
		&\int_0^M \Delta (t-\Delta)^{-1} [g_\veps(x,\cdot)](x) \frac{\diff t}{t^{1-\frac {sp} 2}} \\
		&=
		\int_0^M \int_{|x-y|>\veps} R_t (x-y) \Phi_p(u(x)-u(y))\dd y \frac{\diff t}{t^{1-\frac {sp} 2}}\nonumber\\
		&=\int_{|x-y|>\veps}  \Phi_p(u(x)-u(y)) \left( \int_0^M  R_t (x-y) \frac{\diff t}{t^{1-\frac {sp} 2}}\right) \dd y\\
		&=\int_{|x-y|>\veps}  \frac{\Phi_p(u(x)-u(y))}{|x-y|^{n+sp}} \left( \int_0^{M|x-y|^2}  W(\tau^{\frac{1}{2}}) \frac{\diff \tau}{(\tau^{\frac{1}{2}})^{2-n- {sp} }}\right) \dd y.\nonumber
	\end{align}
	For any $y \ne x$
	\[
	\int_0^{M|x-y|^2}  W(\tau^{\frac{1}{2}}) \frac{\diff \tau}{(\tau^{\frac{1}{2}})^{2-n- {sp} }} \stackrel{M\to \infty}{\longrightarrow}  \int_0^{\infty}  W(\tau^{\frac{1}{2}}) \frac{\diff \tau}{(\tau^{\frac{1}{2}})^{2-n- {sp} }}=:C.
	\]
	We just need to check that $C$ is finite which will ensure we can pass to the limit in the last term of \eqref{eq:BalaklimM} (we can also do it in the first by Step 2(a)) to get \eqref{eq:Balakidentity1} by the Dominated Convergence Theorem since
	\[
	\begin{aligned}
	&\int_{|x-y|>\veps}   \frac{|\Phi_p(u(x)-u(y))|}{|x-y|^{n+sp}} \left( \int_0^{M|x-y|^2}  W(\tau^{\frac{1}{2}}) \frac{\diff t}{(\tau^{\frac{1}{2}})^{2-n- {sp} }}\right) \dd y \\
	&\qquad \leq C \int_{|x-y|>\veps}   \frac{|\Phi_p(u(x)-u(y))|}{|x-y|^{n+sp}}  \dd y<+\infty.
	\end{aligned}
	\]
	
	Let us now compute $C$ explicitly:
	\begin{align*}
		C&=\int_0^{\infty} \left( \frac{\tau^{1-\frac{n}{2}}}{(4\pi)^{\frac{n}{2}}}\int_0^\infty e^{-\tau w} w^{-\frac{n}{2}} e^{-\frac{1}{4w}} \dd w \right)  \frac{\diff \tau}{(\tau^{\frac{1}{2}})^{2-n- {sp} }}\\
		&= \frac{1}{(4\pi)^{\frac{n}{2}}}\int_0^\infty  w^{-\frac{n}{2}} e^{-\frac{1}{4w}}  \left(\int_0^{\infty}e^{-\tau w} \tau^{\frac{sp}{2}}   \diff \tau \right)\dd w \\
		&= \frac{1}{(4\pi)^{\frac{n}{2}}}\int_0^\infty  w^{-1-\frac{n}{2}-\frac{sp}{2}} e^{-\frac{1}{4w}}  \left(\int_0^{\infty}e^{-\eta } \eta^{\frac{sp}{2}}   \diff \eta \right)\dd w\\
		&=\Gamma\left (1+\tfrac{sp}{2}\right)  \frac{1}{(4\pi)^{\frac{n}{2}}}\int_0^\infty  w^{-1-\frac{n}{2}-\frac{sp}{2}} e^{-\frac{1}{4w}} \dd w \\
		&= \Gamma\left(1+\tfrac{sp}{2}\right)  \frac{4^{\frac{sp}{2}}}{\pi^{\frac{n}{2}}}\int_0^\infty  \rho^{-1+\frac{n}{2}+\frac{sp}{2}} e^{-\rho} \dd \rho \\
		&=  \frac{2^{sp}}{\pi^{\frac{n}{2}}}  \Gamma\left(1+\tfrac{sp}{2}\right) \Gamma\left(\tfrac{n+sp}{2}\right).
	\end{align*}
	Clearly $C<+\infty$, which finishes the proof of Step 2.
	
	\textbf{Step 3.} Lastly, we pass to the limit as $\veps\to0$ in \eqref{eq:Balakidentity1}. The limit of the right hand side term was already shown in Step 1 of the proof of \Cref{thm:semigroup}. For the left hand side,
	when $p\geq2$, we have by \Cref{lem:tech1} (we can proceed analogously with \Cref{lem:tech3} and \Cref{lem:tech2} when $p\in(1,2)$)
	we compute directly
	\begin{align*}
		&\Bigg|\int_0^\infty \Delta (t-\Delta)^{-1} [g_\veps(x,\cdot)](x) \frac{\diff t}{t^{1-\frac {sp} 2}}\\
		&\qquad - \int_0^\infty \Delta (t-\Delta)^{-1} [\Phi_p(u(x)-u(\cdot))](x) \frac{\diff t}{t^{1-\frac {sp} 2}}\Bigg|\\
		&\leq
		\int_0^\infty \left| \int_{|x-y|<\veps} \Phi_p(u(x)-u(y)) R_t(x-y) \dd y  \right| \frac{\diff t}{t^{1-\frac {sp} 2}}\\
		&\leq 2 \Phi_p(\|u\|_{L^\infty(\R^n)})\int_0^1 \left( \int_{|z|<\veps}  R_t(z) \dd z  \right) \frac{\diff t}{t^{1-\frac {sp} 2}} \\
		&\qquad +
		C
		\int_1^\infty \left( \int_{|z|<\veps}  R_t(z)|z|^p \dd z  \right) \frac{\diff t}{t^{1-\frac {sp} 2}}.
	\end{align*}
	Clearly the left integral in the last equation goes to $0$ as $\veps\to0$ since $f_\veps(t):= \int_{|z|<\veps}  R_t(z) \dd z  \to0$ pointwise as $\veps \to0$ due to the fact that $\|R_t\|_{L^1(\R^n)}=1$. For the right integral, we have
	\[
	\int_1^\infty \left( \int_{|z|<\veps}  R_t(z)|z|^p \dd z  \right) \frac{\diff t}{t^{1-\frac {sp} 2}}=  \int_1^\infty  \left( \int_{|y|<\veps t^{\frac{1}{2}}}  W(|y|) |y|^p \dd y\right) \frac{\dd t}{t^{1+\frac {(1-s)p} {2}}}
	\]
	which again goes to 0 as $\veps \to 0$ since $ \int_{|y|<\veps t^{\frac{1}{2}}}  W(|y|) |y|^p \dd y \to0$ a.e.\ in $t$ as $\veps \to0$. This finishes the proof.
\end{proof}

\section{The choice of a suitable constant}\label{sec:constants}

The choice of the constant $C_1>0$ in \eqref{eq.spl1} is indifferent in most cases and can be put to 1, but a precise value $C_1(n,s,p)$ can be important in writing some exact formulas or in the limit cases $s\to 0^+$, $s\to 1^-$, or $p\to 2$. Thus, in the linear case $p=2$ the natural value is fixed by the assumption that the Fourier transform of the operator is a power multiplier,
$\mathcal F [(-\Delta)^s u] = |\xi|^{2s} \mathcal F[u]$, cf. \cite{Stein1970}.
Working out the details, the literature points out the value
\begin{equation}\label{const.p2}
	c_{n,s} = \frac{2^{2s} \Gamma \left( \frac{n+2s}{2} \right)  }{ \pi^{\frac n 2} |\Gamma (-s)| }=
	s(1-s)\frac{2^{2s} \Gamma \left( \frac{n+2s}{2} \right)  }{ \pi^{\frac n 2} \Gamma (2-s) }.
\end{equation}
However, this choice may not seem to be complete when $p\not=2$. One might expect the choice of $C_1$ to depend on $p$ in such a way that the expected limits are recovered:
\begin{equation}
	\label{eq:limits}
	\begin{aligned}
	\fpl \phi &\stackrel{s\to 1^-}{\longrightarrow} -\Delta_p\phi(x)\quad \textup{for all } p\in[2,\infty) \\
	\fpl \phi &\stackrel{p\to2^+}{\longrightarrow} (-\Delta)^s\phi(x) \quad \textup{for all } s\in(0,1).
	\end{aligned}
\end{equation}
The choice \eqref{const.p2} satisfies the second limit, but not the first one. We propose the following constant.
\begin{lemma}\label{lem:constant}
	Consider the constant
	\[
	C_1(n,s,p)=\frac{\frac{sp}{2}(1-s)2^{2s-1}}{\pi^{\frac{n-1}{2}}} \frac{\Gamma(\frac{n+sp}{2})}{\Gamma{(\frac{p+1}{2})}\Gamma(2-s)}
	\]
	in \eqref{eq.spl1}. Then, for $\phi \in C^2_b(\R^n)$  we have that \eqref{eq:limits} holds.
\end{lemma}
\begin{proof}
	Consider the operator
	\[
	\mathcal{L}^{s}_p[\phi](x)= \PV \int_{\mathbb R^d} \frac{\Phi_p (u(x) - u(y))}{|x-y|^{n+sp}} \diff y.
	\]
	We will use the following limit
	\begin{equation}\label{eq:lim_s}
		(1-s) \mathcal{L}^s_p[\phi](x)\stackrel{s\to 1^-}{\longrightarrow} - \frac{\pi^{\frac{n-1}{2}}}{p} \frac{\Gamma(\frac{p+1}{2})}{\Gamma(\frac{n+p}{2})}  \Delta_p\phi(x).
	\end{equation}
	When $\nabla\phi(x)\not=0$, \eqref{eq:lim_s} is precisely \cite[Theorem 2.8]{Bucur2020}. The situation when  $\nabla \phi(x)=0$ is even simpler. The case $p=2$ is classical and we omit it (see for example \cite{DiNezza2012}). When $p>2$ we rely on the following intermediate step in the proof of  \cite[Theorem 2.8]{Bucur2020} (see also the proof of \cite[Theorem 6.1]{delTesoLindgren2020}):
	\[
	(1-s) \mathcal{L}^s_p[\phi](x)\stackrel{s\to 1^-}{\longrightarrow} -\frac{p-1}{2p}\int_{\partial B_1(0)} \left|\nabla \phi(x) \cdot y \right|^{p-2} y^{T} D^2\phi(x) y \dd \sigma(y).
	\]
	Clearly the right hand side term in the equation above is zero when $\nabla \phi(x)=0$. Finally, if $\nabla \phi(x)=0$, it is straightforward to see that $-\Delta_p\phi(x)=0$ for $p>2$. This proves \eqref{eq:lim_s}.
	Then,
	\begin{align*}
		&C_1(n,s,p) \mathcal{L}^s_p[\phi](x) = \left(\frac{\frac{sp}{2}2^{2s-1}}{\pi^{\frac{n-1}{2}}} \frac{\Gamma(\frac{n+sp}{2})}{\Gamma{(\frac{p+1}{2})}\Gamma(2-s)} \right) \left((1-s)  \mathcal{L}^s_p[\phi](x) \right)\\
		&\stackrel{s\to 1^-}{\longrightarrow}  \left(\frac{\frac{p}{2}2}{\pi^{\frac{n-1}{2}}} \frac{\Gamma(\frac{n+p}{2})}{\Gamma{(\frac{p+1}{2})}\Gamma(1)} \right) \left(- \frac{\pi^{\frac{n-1}{2}}}{p} \frac{\Gamma(\frac{p+1}{2})}{\Gamma(\frac{n+p}{2})}  \Delta_p\phi(x)\right)=-\Delta_p\phi(x),
	\end{align*}
	where we have used the fact that $\Gamma(1)=1$. This proves the first limit in \eqref{eq:limits}. For the second limit, we just need to check that $C_1(n,s,2)$ coincides with $c_{n,s}$ given in \eqref{const.p2}. This is just a simple computation
	\begin{align*}
		C_1(n,s,2)&=\frac{s(1-s)2^{2s-1}}{\pi^{\frac{n-1}{2}}} \frac{\Gamma(\frac{n+2s}{2})}{\Gamma{(\frac{3}{2})}\Gamma(2-s)}\\
		&= \frac{s(1-s)2^{2s}}{\pi^{\frac{n}{2}}} \frac{\Gamma(\frac{n+2s}{2})}{\Gamma(2-s)}
	\end{align*}
	where we have used the fact that $\Gamma(1/2)=\sqrt{\pi}$ and $z\Gamma(z)=\Gamma(z+1)$. This finished the proof.
\end{proof}

\begin{remark}
	With the choice of $C_1(n,s,p)$ given by \Cref{lem:constant} we have that
	\[
	C_2(s,p):= C_1(n,s,p) \frac {\pi^{ \frac n 2 } }  {2^{sp} \Gamma(\frac{n+sp}2)} =  \pi^{ \frac 1 2 } \frac{p}{\Gamma(\frac{p+1}{2})}\frac{1}{|\Gamma(-s)|} 2^{s(2-p)-2},
	\]
	\[
	C_3(s,p)=C_1(n,s,p)  \frac{ \pi^{\frac n 2} \Gamma (\frac{sp}2)} {\Gamma ( \frac {n + sp}2 )} = \pi^{\frac 1 2} \frac{p}{\Gamma{(\frac{p+1}{2})}} \frac{  \Gamma (\frac{sp}2)} {|\Gamma(-s)|} 2^{2s-2},
	\]
	and
	\begin{align*}
	C_4(s,p)&= C_1(n,s,p) \frac{ \pi^{\frac n 2} } {2^{sp}\Gamma(\frac{2+sp}{2})\Gamma ( \frac {n + sp}2 )} \\
	&= \pi^{\frac{1}{2}}\frac{p}{(\frac{p+1}{2})}  \frac{ 1 } {|\Gamma(-s)|\Gamma(\frac{2+sp}{2})}   2^{s(2-p)-2}.
	\end{align*}
	An important fact is that  these  constants are independent of the dimension. It is also standard to check that
	\[
	\begin{gathered}
		C_2(s,2)=\frac{1}{|\Gamma(-s)|}, \quad  C_3(s,2)=\frac{\Gamma(s)}{|\Gamma(-s)|}4^s,   \\
		\textup{and} \quad C_4(s,2)= \frac{ 1 } {|\Gamma(-s)|\Gamma(1+s)}=\frac{\sin(s\pi)}{\pi}
	\end{gathered}
	\]
	which recover the classical constants for the semigroup  formula and the extension problem (see for example \cite{Kwa17}).
\end{remark}

\section{A new spectral-type fractional $p$-Laplacian on domains}
\label{sec:new spectral}	

There are many well known applications and interesting remarks of the semigroup formula and the extension problem for the fractional Laplacian. After the results that we have shown in this paper, a number of applications can be naturally proposed in the context of the fractional $p$-Laplacian. Here is a first application.

Let $\Omega$ be a smooth domain (bounded or otherwise).
The semigroup representation \eqref{eq:semigintro} motivates the definition of a new operator using formula \eqref{eq:semigdomain}, where we replaced $-\Delta$ by $-\Delta_{\Omega}$, the Laplacian with homogeneous Dirichlet boundary conditions. It is clear that $(-\Delta_\Omega)_p^s$ coincides with $\fpl$ \ if \ $\Omega = \mathbb R^n$, and one can even expect some ``continuous'' dependence on $\Omega$.

This is, to our knowledge,  a new example of nonlocal  operator when $\Omega \neq \mathbb R^n$. We will show below it does not coincide with $\fpl$ applied to the extension of  $u$  by $0$ outside $\Omega$. On the other hand, when $p = 2$ we recover the usual Spectral Fractional Laplacian, which is simply the usual spectral power of the Dirichlet Laplacian in a bounded domain as a compact operator;   we refer the reader to \cite{CaTa10, Stinga2010,MR3360740} for modern references.

First, we show it is well-defined for functions defined exclusively in $\Omega$, at least when they are smooth. This will allow to consider a boundary value problem with boundary conditions in $\partial \Omega$, without taking the set $\mathbb R^n \setminus \Omega$ into account.

\begin{theorem}
	\label{lem:spectral is defined}
	Let  $s \in (0,1)$ and  assume that $u\in C^2(\overline \Omega),$ where $\Omega$ is a smooth subdomain on $\Rd$. If $p \ge \frac{2}{2-s}$, the following  formula
	\begin{equation}\label{eq:semigdomain2}
		(-\Delta_\Omega)_{p}^s u (x_0) :=C_2\int_0^{+\infty} e^{t \Delta_\Omega} [\Phi_p (  u(x_0) - u(\cdot) )](x_0) \frac{\diff t}{t^{1+\frac {sp} 2}}
	\end{equation}
	is well-defined at every $x_0 \in \Omega$. If  $p$ is close to 1, more precisely $p \in (1,\frac{2}{2-s})$, we need to assume additionally that $\nabla u(x_0)\not=0$.
\end{theorem}

\begin{proof}
	Since $\Omega$ is smooth, we can extend $u$ to a $C_c^2(\mathbb R^n)$, via the tubular neighbourhood theorem. We will still denote the extended function by $u$. 		The time integral in  formula \eqref{eq:semigdomain2} could be singular at both ends, $t = 0, \infty$, hence we must show the convergence of the integral at either end. 		The integral over the interval $[1,\infty)$ is not a problem, since $\| e^{t \Delta_\Omega} f \|_{L^\infty (\Omega)} \le C e^{-\lambda_1 t} \| f \|_{L^\infty (\Omega)}$ where $\lambda_1$ is the first eigenvalue of the Dirichlet Laplacian.
	
	Let us show that there exists
	\begin{equation*}
		\lim_{m \to \infty} \int_{1/m}^{1} e^{t \Delta_\Omega} [\Phi_p (  u(x) - u(\cdot) )](x) \frac{\diff t}{t^{1+\frac {sp} 2}}.
	\end{equation*}
	We do this by following the proof done above with $\fpl$, with suitable modifications.
	For convenience, we write the semigroup $e^{t \Delta_{\Omega}}$ in its kernel form $K_\Omega$, and compare with the heat kernel of the whole space $K_{\mathbb R^n}$. We define $g(x,y) = \Phi_p (  u(x) - u(y) )$. Then,
	\begin{equation*}
		e^{t \Delta_\Omega} [\Phi_p (  u(x) - u(\cdot) )](x) = \int_\Omega K_\Omega(t,x,y) g(x,y) \diff y  .
	\end{equation*}
	We write
	\begin{equation}
		\label{eq:sp lap in domain decomposition}
		\begin{aligned}
			&\int_{1/m}^1 \int_\Omega K_\Omega(t,x,y) g(x,y) \diff y  \frac{\diff t}{t^{1+\frac {sp} 2}}\\
			&=\int_{1/m}^1 \int_{\mathbb R^n} K_{\mathbb R^n}(t,x,y) g(x,y) \diff y  \frac{\diff t}{t^{1+\frac {sp} 2}} \\
			&\qquad - \int_{1/m}^1  \int_{\mathbb R^n \setminus \Omega}  K_{\mathbb R^n}(t,x,y)   g(x,y) \diff y  \frac{\diff t}{t^{1+\frac {sp} 2}} \\
			&\qquad - \int_{1/m}^1 \int_\Omega ( K_{\mathbb R^n}(t,x,y)  - K_\Omega (t,x,y)) g(x,y) \diff y  \frac{\diff t}{t^{1+\frac {sp} 2}} .
		\end{aligned}
	\end{equation}
	For the first term on the right-hand side we can compute the limit, which is well-defined by our knowledge of $\fpl$, notice that
	by \eqref{eq:convergence semigroup} in the proof \Cref{thm:semigroup} we have
	\begin{equation*}
		\fpl u(x) = \lim_{m \to \infty} \int_{1/m} ^\infty \int_{\mathbb R^n } K_{\mathbb R^n} (t,x,y) g(x,y) \diff y \frac{\diff t}{t^{1+\frac {sp} 2}}.
	\end{equation*}
	For the other two terms in \eqref{eq:sp lap in domain decomposition}, we will show  the convergence by checking they are Cauchy sequences. To this effect, we write for $k > m > 0$ that
	\begin{equation*}
		\int_{1/m}^1 f(t) \diff t -  	\int_{1/k}^1 f(t) \diff t = \int_{1/m}^{1/k} f(t) \diff t.
	\end{equation*}
	We need to show that the two remaining terms give Cauchy sequences, so we take $ k > m > 0$.
	The second term is easy to estimate
	\begin{align*}
		&\left| \int_{1/k}^{1/m} \int_{\mathbb R^n \setminus \Omega} K_{\mathbb R^n} (t,x,y) g(x,y)\dd y \frac{\diff t}{t^{1+\frac {sp} 2}} \right|  \\
		& \le \| g(x, \cdot) \|_{L^\infty (\Omega)} \int_{1/k}^{1/m}  \frac{1}{(4\pi t)^{\frac n 2} }  \int_{\dist (x, \partial \Omega)}^\infty e^{-\frac{r^2}{4t}} r^{n-1}\diff r \frac{\diff t}{t^{1+\frac {sp} 2}} \\
		&=  \| g(x, \cdot) \|_{L^\infty (\Omega)}  \int_{1/k}^{1/m} \int_{C / t }^\infty e^{-1/\rho} p(\rho)  \diff \rho \, q(t) \diff t \\
		&\le C \| g(x, \cdot) \|_{L^\infty (\Omega)}  \int_{0}^{1/m} e^{-D/t} q(t) \diff t\,,
	\end{align*}
	where $p$ and $q$ are bounded by positive and negative powers and $C, D > 0$. Since $e^{-D t} q(t) \in L^1(0,1)$, the second term of the right-hand side of \eqref{eq:sp lap in domain decomposition} is a Cauchy sequence.
	
	For the last term of \eqref{eq:sp lap in domain decomposition}, using the comparison principle we have that
	\begin{equation*}
		0 \le  K_{\mathbb R^n} (t,x,y) - K_\Omega (t,x,y) \le \max_{z \in \partial \Omega}  K_{\mathbb R^n} (t,x,z) = \frac{1}{(4 \pi t)^{\frac n 2} }e^{- \frac{\dist(x, \partial \Omega)^2}{4t} }.
	\end{equation*}
	Thus,
	\begin{align*}
		&\int_{1/k}^{1/m} \left| \int_{\Omega} ( K_{\mathbb R^n} (t,x,y) - K_\Omega (t,x,y) ) g(x,y) dy \right| \frac{\diff t}{t^{1+\frac {sp} 2}} \\
		&\le \| g (x, \cdot) \|_{L^\infty (\Omega) } \int_{0}^{1/m} \frac{1}{(4 \pi t)^{\frac n 2} }e^{- \frac{\dist(x, \partial \Omega)^2}{4t} }\frac{\diff t}{t^{1+\frac {sp} 2}} .
	\end{align*}
	Hence, the last term of \eqref{eq:sp lap in domain decomposition} is also a Cauchy sequence.
\end{proof}

\begin{remark}
	This operator is indeed different from the usual one on domains, which uses the formula of $\fpl$ restricted to functions $u$ defined in $\Rd$ such that $u = 0$ outside $\Omega$ (see, for example, \cite{Mazon2016,Vazquez2016}). Let $\Omega$ be fixed and assume without loss of generality that $0 \in \Omega$. Let $B_R (0) \subset \Omega$ and $u$ a smooth function that is positive,  radially symmetric, and strictly decreasing along the radius and $0$ outside of $B_R$. Then, with the notation of the proof of \Cref{lem:spectral is defined}
	\begin{align*}
		&\fpl u(0) - (-\Delta_\Omega)_p^s u(0) \\
		&\quad = \int_0^\infty \int_\Omega (K_{\mathbb R^n} (t,0,y) - K_\Omega (t,0,y)) \Phi_p(u(0)-u(y)) \diff y \frac{\diff t}{t^{1+\frac {sp} 2}} \\
		&\qquad + \int_0^\infty \int_{\mathbb R^n \setminus \Omega} K_{\mathbb R^n} (t,x,y)  \Phi_p(u(0)-u(y)) \diff y \frac{\diff t}{t^{1+\frac {sp} 2}}     		> 0.
	\end{align*}
	This happens because both integrands are nonnegative functions, and $K_{\mathbb R^n} (t,x,y) - K_\Omega (t,x,y)\geq0$ for all $x,y\in\Omega$.
\end{remark}

\begin{remark}
	We can similarly define operators via the extension problem and Balakrishnan's formulas. For example,
	\begin{equation*}
		(-\Delta_\Omega)_{p}^s u (x_0) = C_4 \int_0^\infty \Delta (t-\Delta_\Omega)^{-1} [\Phi_p (  u(x_0) - u(\cdot) )](x_0) \frac{\diff t}{t^{1-\frac {sp} 2}}.
	\end{equation*}
	It is an interesting open problem to check whether these definitions are equivalent. Another formula can be written via the extension.
\end{remark}

\section{Numerical applications}\label{sec:numApp}

Among the many possible applications of our representations, we want to comment next on the construction of finite-difference methods for the fractional $p$-Laplacian. We recall that monotone finite-difference discretizations for the fractional Laplacian (case $p=2$) have been established in the papers  \cite{Ci-etal18} (in $\R$) and \cite{DelTeso2018} (in $\R^n$) based on the semigroup representation. The advantage of  these types of discretizations is that they have a quadratic  error order, independently of the fractional power of the operator (see \cite{DelTeso2018}). This is a great advantage compared to other finite difference discretizations in the literature (see discussion in \cite{DelTeso2018, DelTeso2019}).

Similar ideas could be used in the case of the fractional $p$-Laplacian using the results of \Cref{thm:semigroup}. 	
Let us explain the process at a formal level.
Let  $h>0$ and $\{e_i\}_{i=i}^n$ be standard basis of $\R^n$. Consider the discrete Laplacian given by
\begin{equation}\label{eq:disclap}
	\Delta_h [u](x) = \sum_{i=1}^n \frac{u(x+he_i) + u(x-he_i) - 2u(x)}{h^2}.
\end{equation}
The already-classical theory shows that this is an ``asymptotically good'' approximation of $\Delta$ (in the sense that it produces convergent and stable schemes in many settings).
In accordance to our representation ideas, we are led to define the following discrete operator
\begin{equation}\label{eq:discplap}
	(-\Delta_h)_p^s [u] (x) :=  C_2  \int_0^\infty e^{t \Delta_h} [\Phi_p (  u(x) - u(\cdot))](x) \frac{\diff t}{t^{1+\frac {sp} 2}},
\end{equation}
by simply replacing $\Delta$ by $\Delta_h$ in \eqref{eq:semigintro}. In this setting, $e^{t \Delta_h} [f]$ represents the uniquely-defined solution of $\partial_t w = \Delta_h w$ such that $w(x,0) = f(x)$.
We will call it the \emph{discrete fractional $p$-Laplacian}.

\bigskip

At least formally, we have an alternative representation of $(-\Delta_h)_p^s$ given in a finite-difference form similar to \eqref{eq:disclap}.
We follow here the presentation of \cite{DelTeso2018}. Let $z_\beta := \beta h$ for $\beta:=(\beta_1,\ldots,\beta_n) \in \Z^n$ and $\psi:\R^n \to \R$. We have that
\[
e^{t \Delta_{ h}}[\psi](x) = \sum_{\beta\in \Z^n} \psi(x+z_\beta) G(\beta, \tfrac{t}{h^2})
\]
where $G(\beta,t)=e^{-2nt} \Pi_{i=1}^n I_{|\beta_i|}(2t)$ and $I_m$ is the modified Bessel function of first kind with order $m\in \N$. From \eqref{eq:discplap} we get, at least formally,
\begin{align*}
	(-\Delta_h)_p^s u(x)&=C_2  \int_0^\infty \left(  \sum_{\beta\not=0} \Phi_p \left(  u(x) - u(x+z_\beta) \right)G(\beta, \tfrac{t}{h^2}) \right) \frac{\diff t}{t^{1+\frac {sp} 2}}\\
	&= \sum_{\beta\not=0} \Phi_p \left(  u(x) - u(x+z_\beta) \right) C_2 \int_0^\infty G(\beta, \tfrac{t}{h^2}) \frac{\diff t}{t^{1+\frac {sp} 2}}.
\end{align*}
Finally, the change of variables $\tau=t/h^2$ allows to write
\begin{equation}\label{eq:FDpLap}
	\begin{aligned}
	&(-\Delta_h)_p^s u(x)= \sum_{\beta\not=0} \Phi_p \left(  u(x) - u(x+z_\beta) \right) K_{\beta,h} \\
	&\qquad  \textup{with} \quad K_{\beta,h} = \frac{C_2}{h^{sp}} \int_0^\infty G(\beta,t) \frac{\dd t}{t^{1+\frac{sp}{2}}}.
	\end{aligned}
\end{equation}
This process is rigorous when $sp<2$. However, the weights $K_{\beta,h}$ are not well defined when $sp\geq2$:  for example when $n=1$, we note that
\[
G(1,t)=e^{-2nt} I_{1}(2t) \sim t \quad \textup{for} \quad t\ll 1,
\]
which is only enough to show the convergence of the integral in the mentioned range.  However, when $sp\geq2$, it is easy to show that formula \eqref{eq:discplap} fails to have a finite value at a point $x$ of global strict maximum of $u$. 			
An idea to solve this problem is introducing an extra discretization parameter $\delta >0$ and alternatively define
\begin{align}\label{eq:discFPL2}
	(-\Delta_{h,\delta})_p^s:=&  C_2  \int_\delta ^\infty\left(e^{t \Delta} [\Phi_p (  u(x) - u(\cdot) \right)](x) \frac{\diff t}{t^{1+\frac {sp} 2}}\\\nonumber
	=& \sum_{\beta\not=0} \Phi_p \left(  u(x) - u(x+z_\beta) \right) K_{\beta,h,\delta }, \qquad \qquad \\
	\textup{with}  & \quad K_{\beta,h,\delta } = \frac{C_2}{h^{sp}} \int_{\delta h^2}^\infty G_{ h}(\beta,t) \frac{\dd t}{t^{1+\frac{sp}{2}}}.
\end{align}
In this way, all the integral are finite and all the steps are well justified rigorously in the whole range $p\in (1,\infty)$ and $s\in(0,1)$. One should make $\delta \to 0$ as $h \to 0$ to show that the discretization is consistent in some proper sense.

\begin{remark}
	A different finite difference discretization of the from of \eqref{eq:FDpLap} for nonlocal $p$-Laplacian type operators (in the sense of \cite{Andreu2010})  has been obtained in \cite{Perez11}. They are based on obtaining quadrature rules for a singular integral definition like \eqref{eq.frlap}.  In both cases the obtained discretizations are monotone, which yields good properties of the associated schemes.
\end{remark}

\begin{remark}
	Balakrishnan's formula has been the starting point for one of the most popular numerical schemes for the fractional Laplacian. The scheme combines a quadrature for the time integral and a numerical scheme for a sequence of the spatial operators $\Delta(t-\Delta)^{-1}$. One of the main advantages of this idea is the fact that, contrary to others, it is suited for parallel computing. We refer to \cite{BoPa15} as the basic reference to the topic.
\end{remark}

\begin{remark}
	The extension problem has also been successfully used by the numerical analysis community in the context of the fractional Laplacian. We refer the reader to \cite{Teso14} for a finite difference treatment in the context of nonlocal porous medium equations and to \cite{NoOtSa15} for a finite element one in the context of boundary value problems.
\end{remark}

\section{Other applications}\label{sec:app}

We collect here a series of directions in which our previous ideas can be developed.

\subsection{Equivalent representations of the $W^{s,p}(\R^n)$ seminorm}\label{ssec.8.1}
We can define an equivalent way of defining the $(s,p)$-energy \eqref{Jsp1}. More precisely, consider the Gagliardo semi-norm
\[
[u]_{W^{s,p}(\R^n)}:= \left(C_1\int_{{\mathbb R^d}}\int_{{\mathbb R^d} } \frac{|u(x)-u(y)|^p}{|x-y|^{N+sp}} \dd x \dd y \right)^{\frac{1}{p}}.
\]
It is standard to check that we have the following alternative definition
\[
[u]_{W^{s,p}(\R^n)}= \left(C_2\int_{\R^n}\int_0^{+\infty} e^{t \Delta} [ |u(x) - u(\cdot)|^p ](x) \frac{\diff t}{t^{1+\frac {sp} 2}}\dd x\right)^{\frac{1}{p}}.
\]
This formulation was introduced in \cite{Tai64}. We also refer to \cite{Garofalo2020} for a modern reference on the topic in the context of Carnot groups.

The corresponding version with Balakrishnan's formula, is
\[
[u]_{W^{s,p}(\R^n)}=\left(C_4\int_{\R^n} \int_0^\infty \Delta (t-\Delta)^{-1} [|u(x) - u(\cdot)|^p](x) \frac{\diff t}{t^{1-\frac {sp} 2}} \dd x\right)^{\frac{1}{p}}.
\]
We have not found this expression in the literature.
A similar formula can be constructed via the extension, but it comes as an awkward product of gradients.

\subsection{Spectral $(s,p)$-Laplacian on manifolds}

Let $(M,g)$ be a Riemannian manifold. Some authors (see, e.g. \cite{Guo2018}) have used the $(s,p)$-fractional Laplacian on manifolds given by
\begin{equation*}
	(-\Delta)_{M,p}^s u(x) = C_3(n,s,p) \PV \int_{M} \frac{\Phi_p (u(x) - u(y))}{d_g(x,y)^{n+sp}} \diff y,
\end{equation*}
where $g$ is a Riemannian metric on $M$. This a regional-type operator.
Following our approach for bounded domain, we propose also the spectral-type operator
\begin{equation*}
	(-\Delta_M)_p^s u (x) =
	C_1 (n,s,p)
	\int_0^{+\infty} e^{- t \Delta_g} \Phi_p (  u(x) - u(\cdot) ) \frac{\diff t}{t^{1+sp}}
\end{equation*}
where $\Delta_g$ is the Laplace-Beltrami operator on $M$. The change in sign with respect to the other formulas is due to the convention in Geometry, where $\Delta_g$ is the positive operator.

The idea of dealing with this kind of operators has been considered in the case of the fractional versions of the heat equation  and porous medium equations.  The extension method was used by Banica et al. \cite{Banica2015} for the fractional heat equation on hyperbolic space or non compact manifolds. The spectral approach was used quite recently by Berchio et al. \cite{MR4150871} for the fractional porous medium equation posed in hyperbolic space.

\subsection{Representation via a general parametric family of operators}\label{ssec.gen}
Following the philosophy of \Cref{sec:semigrep} and \Cref{sec:Balak} we can think of more general representation formulas for the fractional $p$-Laplacian. Consider a parametric operator family $\{T_t\}_{t\geq0}$ defined via convolution with measure kernels $K(t,\cdot)$, i.e.,
\begin{equation*}
	T_t f(x) = \int_\Rd K(t,x-y) f(y) \diff y\,.
\end{equation*}
At minimum, we need to assume that
\begin{equation*}
	\int_0^\infty K(t, y) \frac{\diff t}{t^{1+\alpha}} = \frac{ C_K}{|y|^{n+sp}}  \quad \text{for all } y \in \Rd.
\end{equation*}
for some finite constant $C_K\ne 0$. Then, we can expect that the following formula holds
\begin{equation}\label{eq:general semigroup splitting formula}
	\fpl u(x) =
	C_T
	\int_0^{+\infty}  T_t \left[  \Phi_p (  u(x) - u(\cdot) ) \right] (x) \frac{\diff t}{t^{1+\alpha}}
\end{equation}
with $C_T=C_1(n,s,p)/C_K$.
These formulas have a particularly nice form if $K$ is self-similar, like for example
\[
K(t,y) =t^{-\beta} F(t^{-\gamma}|y|) \quad \textup{with} \quad \gamma\not=0 \quad \textup{and} \quad \frac{\beta+\alpha}{\gamma} = n+sp,
\]
so that
\[
C_K  \defeq \frac{1}{\gamma} \int_0^\infty F(r) {r^{n+sp-1}}  {\dd r}.
\]
In some settings, our proofs of the semigroup and the Balakrishnan's formulas can be repeated. However, if $K$ changes sign and it is not $L^1$, some estimates are likely more difficult. This can be the case, for example, when working with $T_t = (-\Delta)^k (t - \Delta)^m$.

\begin{remark}
	Note that when we use the heat semigroup, $T_t=e^{t \Delta}$, we have
	\[
	K(t,y)= t^{-\frac{n}{2}}F(t^{-\frac{1}{2}}|y|) \quad \textup{with} \quad F(r)=\frac{1}{(4\pi)^{\frac{n}{2}}} e^{ -\frac{r^2}{4}}.
	\]
	Here, $\gamma=1/2$, and we must take $\alpha=sp/2$. It works for all $s$ and $p$, see \eqref{repr1}.
\end{remark}

\begin{remark}  Another option is to use a fractional heat semigroup, $T_t = e^{-t (-\Delta)^\sigma}$, with $\sigma \in (0,1)$. Then the representation \eqref{eq:general semigroup splitting formula}  is correct on the condition that $ sp<2\sigma$.
	
	Let us sketch  the proof. It is well known, see for instance \cite{BG60, BSV17}, that the fractional semigroup has a kernel $K$  with scaling parameters $\beta=n/2\sigma$ and $\gamma=1/2\sigma$, and a profile $F$ that decays like $F(r)\sim r^{-(n+2\sigma)}$ as $r\to\infty$. It follows that we must take $\alpha= sp/2\sigma$. In view of the decay of $F$, the integral for $C_K$ is convergent only if $sp<2\sigma$, and then the  technique applies. On the contrary, the fat tail of $F$ is incompatible with this representation if $sp\ge 2$ because then we cannot find a $\sigma<1 $ such that $C_K$ is finite.
\end{remark}

\section{Final comments and open problems}\label{sec:commentandproblems}

\noindent{\bf Intuition behind our results: splitting the formula for the fractional $p$-Laplacian.}
An idea behind the representation formula \eqref{eq:semigintro} is the following simple principle: the fractional $p$-Laplacian  (which is both a nonlinear and nonlocal operator) of a function $u:\R^n\to\R$ can be formally seen as the composition of a linear nonlocal operator acting on $n$ variables of a certain function $v:\R^{2n}\to\R$ which is built from $\Phi_p$ and $u$.	
Indeed, let $v(x,y) = \Phi_p(u(y) - u(x))$. Notice that $v(x,x) = 0$. At least formally, for $sp <  2$, and for convenient functions, we have that
\begin{align*}
	\fpl u (x) = C_1 \int_{ \Rd }\frac{v(x,x) - v(x,y)}{|x-y|^{n+sp}} \diff y = \widetilde{C}_1(-\Delta)^{\frac {sp} 2}[v(x, \cdot)] (x).
\end{align*}
This shows a splitting between the linear part $(-\Delta)^{\frac{sp} 2}$ and the nonlinear part $\Phi_p$. For $sp \ge 2$ one should define the operator
\[
\mathcal{L}^{\frac{sp}{2}}[\psi](x)=C_1\PV \int_{|y|>0} \frac{ \psi(x) - \psi(y)}{|x-y|^{n+sp}}\dd y,
\]
which does not coincide with $(-\Delta)^{\frac {sp} 2 }$. Naturally, this very singular integral is only defined for special $\psi$ and $x$. A key part of our arguments are based on showing that, even though the operator $\mathcal{L}^{\frac{sp}{2}}$ is hardly ever defined, the composition works. The scaling properties of the corresponding kernels and integrals in $t$ or $z$ will be crucial.
This splitting idea is also shared by the two other alternative representations \eqref{eq:extintro} and \eqref{eq:Balaintro}.

\smallskip

\noindent{\bf Other representation formulas.}
For the case $p = 2$, other representation formulas are known (see, e.g. \cite{Kwa17}). It seems likely  that some of them  can be adapted for the case $p \ne 2$, but we have chosen not to deal with that issue in this paper.

\smallskip

\noindent{\bf Optimal domains of definition for the new formulas.} After our proofs of equivalence of the operator definitions when applied to functions with a  required regularity, an interesting issue is to extend the domain of definition of the operators, as hinted in Subsection \ref{ssec.8.1} or done in reference \cite{Kwa17} for the linear case, and then to show that the equivalence still holds, at least in some weak or viscosity version. This is a delicate issue that we have decided to postpone. It could be very important for the possible applications
to stationary or evolution problems.

\smallskip

\noindent{\bf Other operators arising from subordination for nonlinear semigroups. } In the paper \cite{CiGr09}, Cipriani and Grillo proposed another type of nonlinear and nonlocal operators defined via subordination applied directly to nonlinear semigroups. More precisely, given a nonlinear semigroup $T_t$, they define (up to a normalisation constant)
\begin{equation}\label{eq:nonlinSub}
	A_s[u](x):= \int_{0}^\infty \left(u(x) - T_t[u](x) \right) \frac{\dd t}{t^{1+s}}.
\end{equation}
After our semigroup representation formula \eqref{eq:semigintro} for the fractional $p$-Laplacian $(-\Delta)_p^s$ one might wonder if $A_s$ coincides with $(-\Delta)_p^s$ for some semigroup associated to a local operator, as is happens for the case $p=2$ (the fractional Laplacian). Here the natural candidate is the semigroup $T_t^q$ associated to the $p$-Laplacian operator $\Delta_q u:=\nabla \left( |\nabla u|^{q-2} \nabla u\right)$ for $q>1$, i.e., $T_t^q[u](x):=w(x,t)$ where
\begin{equation*}
	\begin{dcases}
		\partial_t w(x,t)- \Delta_q w(x,t)=0, & x \in \Rd , t > 0 , \\
		w (x,0) = u(x), & x \in \Rd .
	\end{dcases}
\end{equation*}
We perform now a scaling-based argument which reveals that the operator $A_s^q$ defined by \eqref{eq:nonlinSub} for $T_t^q$ is indeed different than the operator $(-\Delta)_p^s$ for any choice of $p,q\in(1,\infty)\setminus \{2\}$. Clearly, for any $q\not=2$ and  $h>0$ we have
\[
\Delta_q [h u](x)= h^{q-1} \Delta_q [ u](x) \quad \textup{and} \quad \Delta_q [ u(h\cdot)](x)= h^{q} \Delta_q [ u] (hx).
\]
From here, it is standard to get
\[
T_t^q[hu](x)=h T_{h^{p-2}t}[u](x) \quad \textup{and} \quad T_{t}^q[h^{\frac{q}{q-2}}u(h\cdot)](x) = h^{\frac{q}{q-2}} T_{t}^q[u](hx),
\]
so that
\begin{equation}\label{eq:scalingGrillo}
	A_s^q[hu](x) = h^{1+(q-2)s}A_s^q[u](x) \quad \textup{and} \quad A_s^q[u(h\cdot)](x)= h^{sq} A_s^q[u](hx).
\end{equation}	
On the other hand, direct computations show that
\[
(-\Delta)_p^s[hu](x)=h^{p-1} (-\Delta)_p^s[u](x),
\]
and 
\[
(-\Delta)_p^s[u(h\cdot)](x)= h^{sp} (-\Delta)_p^s[u](hx).
\]
If both operators were the same, we must have $p=2+(q-2)s$ and $p=q$ which is clearly impossible for a given $s<1$ unless $p=q=2$.

\smallskip

\noindent{\bf Uniqueness of the extension problem \eqref{eq:extension2} for $sp\geq2$.} For $sp<2$ the uniqueness of energy solutions of \eqref{eq:extension2}  is shown in \cite{Caffarelli2007}. However, for $sp \ge 2$ the diffusion becomes too singular at $y = 0$ to work properly in the energy setting. It is an open problem to show in which sense of solution $E_{s,p}[f]$ is the unique solution of the boundary value problem.

\smallskip

\noindent{\bf Numerical methods for the spectral-type fractional $p$-Laplacian on domains.}  Following the ideas of \cite{Cusimano18}, one can implement numerical discretization of the fractional $p$-Laplacian on bounded domains proposed in \Cref{sec:new spectral}.

\smallskip

\noindent{\bf Regularity and well-posedness.} The extension techniques have been extensively used in the theory of regularity of nonlocal problems involving linear operators, i.e the case $p=2$ of the theory presented in this paper. We refer for example to the seminal work of Athanasopoulos and Caffarelli \cite{AtCa10} in the context of fractional porous medium and Stefan type problems. It has also been used to prove well-posedness and properties of such type of equations (see \cite{dP11, dP12}).

We believe that our theory could also help to develop similar results for equations involving the fractional $p$-Laplacian.

\appendix
	
	\section{Technical lemmas}
	
	\begin{lemma}\label{lem:tech1}
		Let $u\in C_b^2(\R^d)$ and $\Phi_p(t)=|t|^{p-2}t$ for $p\geq2$. Let $K:\R^d\to \R_+$ be such that $K(z)=K(-z)$ and $\int K(z) |z|^p \dd z < +\infty$. Then
		\begin{align*}
		&\left| P.V. \int_{|x-y|>0} K(x-y) \Phi_p (u(x)-u(y)) \dd y \right| \\
		&\qquad \leq (p-1) \|\nabla u\|^{p-2}_{L^\infty(\Rd)} \|D^2 u\|_{L^\infty(\Rd)} \int_{\R^d} K(z) |z|^p \dd z.
		\end{align*}
	\end{lemma}
	\begin{proof}
		By symmetry
		\begin{align*}
			&\Bigg| \PV \int_{|x-y|>0} K(x-y) \Phi_p (u(x)-u(y)) \dd y \Bigg| \\
			&= \left| \PV \int_{|z|>0} K(z) \Phi_p (u(x)-u(x+z)) \dd z \right|\\
			&=\frac{1}{2}\left| \PV \int_{|z|>0} K(z)\left( \Phi_p (u(x)-u(x+z))+\Phi_p (u(x)-u(x-z))\right) \dd z \right| \\
			&\leq \frac{1}{2}\PV \int_{|z|>0} K(z)\left|  \Phi_p (u(x)-u(x+z))-\Phi_p (-u(x)+u(x-z))\right| \dd z.
		\end{align*}
		Now, since $p\geq2$, we have for some $|\xi| \leq  \sup_{x,z}|u(x)-u(x+z)| \leq \|\nabla u\|_{L^\infty(\Rd)} |z|$
		\begin{align*}
			&\left|  \Phi_p (u(x)-u(x+z))-\Phi_p (-u(x)+u(x-z))\right| \\
			&\qquad  \leq |\Phi_p'(\xi)| |u(x)-u(x+z) - (-u(x)+u(x-z))|\\
			&\qquad=  (p-1) |\xi|^{p-2}|2u(x)-u(x+z) -u(x-z)|\\
			&\qquad\leq (p-1) \|\nabla u\|^{p-2}_{L^\infty(\Rd)} |z|^{p-2} \|D^2 u\|_{L^\infty(\Rd)} |z|^2\\
			&\qquad= (p-1) \|\nabla u\|^{p-2}_{L^\infty(\Rd)} \|D^2 u\|_{L^\infty(\Rd)} |z|^p.
		\end{align*}
		and the proof is done.
	\end{proof}
	
	\begin{lemma}\label{lem:tech3}
		Let $u\in C_b^2(\R^d)$ and $\Phi_p(t)=|t|^{p-2}t$ for $p\in(1,2)$. Let $K:\R^d\to \R_+$ be such that $K(z)=\kappa(|z|)$ and $\int K(z) |z|^{p} \dd z < +\infty$. If $\nabla u(x)\not=0$, then
		\[
		\left| \PV \int_{|x-y|>0} K(x-y) \Phi_p (u(x)-u(y)) \dd y \right| \leq c \int_{\R^d} K(z) |z|^{p} \dd z.
		\]
		for some constant $c$ depending on $p,n$ and $\|u\|_{C^2(\R^n)}$.
	\end{lemma}
	
	\begin{proof}
		We follow  the proof of Lemma 3.6 in \cite{Korvenpaa2019} to show that
		\begin{align*}
			&\Bigg| \PV \int_{|x-y|>0} K(x-y) \Phi_p (u(x)-u(y)) \dd y \Bigg| \\
			&\lesssim \int_{|z|>0} \left( |\nabla u(x)\cdot z| + \|D^2 u\|_{L^\infty(\R^n)} |z|^2\right)^{p-2} \|D^2 u\|_{L^\infty(\R^n)} |z|^2 K(z) \dd z.
		\end{align*}
		Changing to polar coordinates and using Lemma 3.5 in  \cite{Korvenpaa2019} we get
		\begin{align*}
			&\Bigg| \PV \int_{|x-y|>0} K(x-y) \Phi_p (u(x)-u(y)) \dd y \Bigg| \\
			&\lesssim\|D^2 u\|_{L^\infty(\R^n)} \int_{0}^\infty r^2 \kappa(r) \left( \int_{S^n} \left( |\nabla u(x)\cdot \omega| r  + \|D^2 u\|_{L^\infty(\R^n)} r^2\right)^{p-2} \dd \omega \right)\dd r\\
			&= \|D^2 u\|_{L^\infty(\R^n)} |\nabla u(x)|^{p-2} \\
			&\qquad \times \int_{0}^\infty r^p \kappa(r) \left( \int_{S^n} \left( \left|\frac{\nabla u(x)}{|\nabla u(x)| }\cdot \omega\right |   + \frac{\|D^2 u\|_{L^\infty(\R^n)}}{|\nabla u(x)|} r\right)^{p-2}\dd \omega \right)\dd r\\
			&\lesssim \int_{0}^\infty r^p\kappa(r)  \left(1+ \frac{\|D^2 u\|_{L^\infty(\R^n)}}{|\nabla u(x)|} r\right)^{p-2} dr \\
			&\lesssim \int_{0}^\infty r^p\kappa(r)  dr
		\end{align*}
		where in the last inequality we have used the fact that $f(t)=|t|^{p-2}$ is a decreasing function for $p<2$. Changing back to regular coordinates, we get the desired result.
		
	\end{proof}

	\begin{lemma}\label{lem:tech2}
		Let $u\in C_b^2(\R^d)$ and $\Phi_p(t)=|t|^{p-2}t$ for $p\in(1,2)$. Let $K:\R^d\to \R_+$ be such that $K(z)=\kappa(|z|)$ and $\int K(z) |z|^{2p-2} \dd z < +\infty$. Then
		\[
		\left| \PV \int_{|x-y|>0} K(x-y) \Phi_p (u(x)-u(y)) \dd y \right| \leq c \int_{\R^d} K(z) |z|^{2p-2} \dd z.
		\]
		for some constant $c$ depending on $p,n$ and $\|u\|_{C^2(\R^n)}$.
	\end{lemma}
	
	\begin{proof}
		We follow  in the proof of Lema 3.6 \cite{Korvenpaa2019} to show that
		\begin{align*}
			&\Bigg| \PV \int_{|x-y|>0} K(x-y) \Phi_p (u(x)-u(y)) \dd y \Bigg| \\
			&\lesssim \int_{|z|>0} \left( |\nabla u(x)\cdot z| + \|D^2 u\|_{L^\infty(\R^n)} |z|^2\right)^{p-2} \|D^2 u\|_{L^\infty(\R^n)} |z|^2 K(z) \dd z\\
			&\lesssim \|D^2 u\|_{L^\infty(\R^n)}^{p-1}  \int_{|z|>0}  |z|^{2p-2} K(z) \dd z
		\end{align*}
		where in the last inequality we have used the fact that $f(t)=|t|^{p-2}$ is a decreasing function for $p<2$.
		
	\end{proof}
	
	\begin{remark}
		If $K(z)=|z|^{-n-sp}\chi_{|z|\leq1}(z)$ then \Cref{lem:tech2} applies for $p>\frac{2}{2-s}$ since
		\[
		\int K(z) |z|^{2p-2} \dd z\lesssim \int \frac{\dd z}{|z|^{n+sp+2-2p}}
		\]
		that is finite if $sp+2-2p<0$, that is $p>\frac{2}{2-s}$.
	\end{remark}

\section*{Acknowledgements}
F. del Teso was partially supported by PGC2018-094522-B-I00 from the MICINN of the Spanish Government.
The work of D. G\'omez-Castro  and J. L. V\'azquez was funded by  grant PGC2018-098440-B-I00 from  the Spanish Government.
D. G\'omez-Castro was supported by the Advanced Grant Nonlocal-CPD (Nonlocal PDEs for Complex Particle Dynamics: Phase Transitions, Patterns and Synchronization) of the European Research Council Executive Agency (ERC) under the European Union’s Horizon 2020 research and innovation programme (grant agreement No. 883363).
J.~L.~V\'azquez is an Honorary Professor at Univ.\ Complutense de Madrid. We thank Bruno Volzone and the referees for interesting suggestions.

\medskip

 \bigskip

 \smallskip

 \it

 \noindent
$^1$ Dpto.\ de Análisis Matemático y Matemática Aplicada, \\
Universidad Complutense de Madrid. \\
email: \href{mailto:felix.delteso@ucm.es}{fdelteso@ucm.es} \\[4pt]
$^2$ Mathematical Institute, \\
 University of Oxford. \\ email: \href{mailto:david.gomezcastro@maths.ox.ac.uk}{david.gomezcastro@maths.ox.ac.uk} (Corr. author)\\[4pt]
 $^3$ Departamento de Matemáticas, \\ Universidad Autónoma de Madrid.\\ \href{mailto:juanluis.vazquez@uam.es}{juanluis.vazquez@uam.es}\\

\end{document}